\documentclass[11 pt]{amsart}
\usepackage[all]{xy}
\usepackage{amssymb,amsmath,color}

\newtheorem{theorem}{Theorem}
\newtheorem{defn}[theorem]{Definition}

\newtheorem{lemma}[theorem]{Lemma}

\newtheorem{cor}[theorem]{Corollary}

\newtheorem{prop}[theorem]{Proposition}
\newtheorem{propdef}[theorem]{Proposition/Definition}
\theoremstyle{remark}

\numberwithin{equation}{section}

\def\l{{\mathfrak l}}

\newcommand{\G}{{\Gamma}}

 \let\a\alpha  \let\b\beta    
  
  \let\l\lambda   
      
\let\GL\Lambda
\def\C{\mathbb C}
\def\N{\mathbb N}
\def\G{\Gamma}

\def\GL{\mathbf{GL}}
\def \GL2 {{\text{GL}_2}}

\def\Z{{\mathbb Z}}

\def\Q{{\mathbb Q}}

\def\G{\Gamma}

\newcommand*\HYPERskip{&}
\catcode`,\active
\newcommand*\pFq{
\catcode`\,\active
\def ,{\HYPERskip}%
\doHyper
}
\catcode`\,12

\def\doHyper#1#2#3#4#5{
\, _{#1}F_{#2}\left[{{\begin{array}{ccccccccccc}#3 \end{array}} \atop \begin{matrix}#4 \end{matrix}}; #5 \right] }

\title{Some supercongruences occurring in truncated hypergeometric series}
\author{Ling Long}
\address{Department of Mathematics, Louisiana State University, Baton Rouge, LA 70803, USA}
\email{llong@math.lsu.edu}
\author{Ravi Ramakrishna}
\address{Department of Mathematics, Cornell University, Ithaca, NY 14853, USA}
\email{ravi@math.cornell.edu}
\thanks{The first author is supported by the NSF grant (\#DMS 1303292).  This project was started when she was a Michler fellow at Cornell University 2012-2013. She would like to thank both the Association for Women in Mathematics and Cornell University for the opportunity. The second author
thanks the Tata Institute of Fundamental Research for its hospitality
while this paper was completed.}

\begin{document}
\maketitle
\begin{abstract}
For the purposes of this paper {\it supercongruences} are congruences between terminating hypergeometric
series and quotients of $p$-adic Gamma functions that are stronger than those one can expect to prove
using commutative formal group laws. We prove a number of such supercongruences by using classical hypergeometric
transformation formulae. These formulae (see the appendix), most of which are decades or centuries old,
allow us to write the terminating series as the ratio of products of
of $\G$-values. At this point
{\it sums} have become  {\it quotients}.  Writing these $\G$-quotients as
$\Gamma_p$-quotients,  we are in a situation that is well-suited for proving
$p$-adic congruences.
These $\G_p$-functions can be $p$-adically approximated by their Taylor series expansions.
Sometimes there is cancelation of the lower order terms, leading to stronger congruences.
Using this technique we prove, among other things, a conjecture of Kibelbek and a strengthened version
of a conjecture of van Hamme.
\end{abstract}
\section{Introduction}

Set $\displaystyle (a)_k=a(a+1)\cdots(a+k-1) = \frac{\G(a+k)}{\G(a)}$, the rising factorial or the Pochhammer symbol, where $\G(x)$ is the Gamma function.
For $r$ a nonnegative integer and $\alpha_i,\beta_i \in \mathbb{C}$ with $\beta_i$, the (generalized)
hypergeometric series $_{r+1}F_{r}$  defined by
\[
\pFq{r+1}{r}{\alpha_1, \ldots , ,\alpha_{r+1}}{ ,\beta_1  , \ldots , \beta_r}{\l} :=
\sum_{k= 0}^{\infty} \frac{(\alpha_1)_k(\alpha_2)_k \ldots (\alpha_{r+1})_k}{(\beta_1)_k
\ldots (\beta_r)_k} \cdot \frac{\l^k}{k!}
\]
converges for $|\l|<1$ {    if it is well-defined.
When any  of the
$\a_i$ is a negative integer and none of the $\beta_i$'s is a negative integer larger than $\a_i$, the above sum terminates.}  We set
$$\pFq{r+1}{r}{\alpha_1, \ldots , ,\alpha_{r+1}}{ ,\beta_1  , \ldots , \beta_r}{\l} _n :=
\sum_{k= 0}^{n} \frac{(\alpha_1)_k(\alpha_2)_k \ldots (\alpha_{r+1})_k}{(\beta_1)_k \ldots (\beta_r)_k}
\cdot \frac{\l^k}{k!}
$$
the truncation of the series after the $\l^n$ term.

Hypergeometric series are of fundamental
importance in many research areas including algebraic
varieties, differential equations, Fuchsian groups  and modular forms. For instance,
periods of abelian varieties such as elliptic curves, certain K3 surfaces and other Calabi-Yau manifolds  can be described
by hypergeometric series (\cite{BvS}). Indeed,
the  Euler integral representation of $_2F_1$ (Theorem $2.2.1$ of \cite{AAR})
\begin{equation}\label{eq:Euler}
p_{\{a,b;c\}}(\l):=\int_{0}^1 t^{b-1}(1-t)^{c-b-1}(1-\l t)^{-a}dt=\frac{\G(b)\G(c-b)} {\G(c)}\cdot \,\pFq{2}{1}{a,b}{c}{\l}
\end{equation}
holds when  $\Re c>\Re b>0$ and $\l \in \mathbb C\backslash [1,\infty)$.
When $a=b=\displaystyle\frac12$ and $c=1$ the right side
$p_{\{\frac 12,\frac 12;1\}}(\l)$  is  a period
of  the Legendre family of elliptic curves $E_\l:y^2=x(x-1)(x-\l)$ parameterized by
$\l$. A generalization of the Euler integral representation to
$_{r+1}F_r$ is presented in \cite[\S4.1]{Slate}. These periods are in general complicated transcendental
numbers.  They are much more predictable  when
the elliptic curve has complex multiplication (CM), e.g. $\l=-1$.
The Selberg-Chowla formula  predicts that any period of
a CM elliptic curve is an algebraic multiple of a  quotient of Gamma values \cite{SC}.
For instance, using a formula of Kummer (see \eqref{eq:Kummer}) one can compute that  $\displaystyle p_{\{\frac 12,\frac 12;1\}}(-1)=\frac{\sqrt 2}{4}\frac{\G\left(\frac14\right)^2}{\G\left(\frac12\right)}$.
The hypergeometric series
{    $\pi^2\cdot \pFq{3}{2}{\frac 12,\frac 12,\frac 12}{1,1}{\l}$ is
period of the $K3$ surface $X_\l: s^2=xy(x-1)(y-1)(x-\l y)$ (see \cite[\S4.1]{Slate}) and \cite{AOP}.}
Note that in equation (1) of \cite{AOP} $\l$ bears a sign opposite to our description
of $X_{\l}$.

This paper focuses on a  $p$-adic analog of  these complex periods computed from
the hypergeometric series and Gamma function.
We  motivate our results using the following example. {    For any prime $p \equiv 1 \mod 4$, the elliptic curve  $E_{-1}$ has ordinary reduction at  $p$.
From the theory of commutative formal group laws (CFGL) it is known
that for any integer $r\ge 1$}
\begin{equation}\label{eq:notsuper}
\pFq{2}{1}{\frac {1-p^r}2,\frac {1+p^r}2}{1}{-1}{\big/}
\,\pFq{2}{1}{\frac {1-p^{r-1}}2,\frac {1+p^{r-1}}2}{1}{-1}
\equiv {   (-1)^{\frac{p^2-1}8}}\, \frac{\G_p\left(\frac14\right)^2}{\G_p\left(\frac12\right)}\mod p^r
\end{equation}where $\G_p(\cdot)$ stands for the $p$-adic Gamma function recalled in \S \ref{ss:Preliminaries}.  This  is closely related to  Dwork's unit root for $E_{-1}$ at ordinary primes in terms of truncated hypergeometric series  \cite{Dwork-pcycle}
 and  the $a_p$-values of the  modular form corresponding to $E_{-1}$.
For more details see \cite{KLMSY}.
Here we are interested in the so-called supercongruences,  congruences
stronger than those  predicted by CFGLs. For instance, using the arithmetic of
elliptic curves and modular functions, Coster and van Hamme showed in
\cite{CvH} the congruence
\eqref{eq:notsuper}  holds$\mod p^{2r}$. This result  can also be recovered by our method.

Various supercongruences have been conjectured  by  many mathematicians including Beukers \cite{Beukers2}, van Hamme
\cite{vanHamme},  Rodriguez-Villegas \cite{RV}, Zudilin\cite{Zudilin},  Chan et al. \cite{Chan}, and many more by Z.-W. Sun \cite[et al.]{Sun1,Sun2}.  Often the statements relate truncated hypergeometric series to  Hecke eigenforms and hence to Galois representations.
Some of these conjectures are proved using a variety of methods, including
Gaussian hypergeometric series \cite[et al.]{Ahlgren, AO,Kil,MO,McO}, the Wilf-Zeilberger method \cite{Zudilin} and  $p$-adic analysis \cite{Beukers1,OSS}.  For more information on recent development on supercongruences, please see \cite{OSS}.

In this paper, we will show the following  supercongruence result at $\l=2$. The elliptic curve $E_2$ has CM and
 is isomorphic to $E_{-1}$ over $\Q$.
\begin{theorem}\label{thm:1}For any prime $p\equiv 1\mod 4$ and any integer $r\ge 1$,
 \begin{equation}{\pFq{2}{1}{\frac{1-p^r}2,\frac {1}2}{1}{2}}\big/{\,\pFq{2}{1}{\frac{1-p^{r-1}}2,\frac {1}2}{1}{2}}\equiv (-1)^{\frac{p^2-1}8} \frac{   \G_p\left(\frac14\right)^2}{\G_p\left(\frac12\right)} \mod p^{2r}.
\end{equation}
\end{theorem}

We outline the strategy of the proof.
The strategies for the proof of Theorems \ref{thm:2} and \ref{thm:3} are outlined in
\S$4$ and $5$. We use the Pfaff transformation and Kummer evaluation formula to obtain
the equalities
$$\pFq{2}{1}{\frac{1-p^r}2,\frac {1}2}{1}{2}=
\frac{\left(\frac12\right)_{\frac{p^r-1}2}}{(1)_{\frac{p^r-1}2}}\,
\pFq{2}{1}{\frac{1-p^r}2, \frac12}{1-\frac{p^r}2}{-1}
= \frac{(\frac12)_{\frac{p^r-1}2}}{(1)_{\frac{p^r-1}2}}\cdot
\frac{\G\left(1-\frac{p^r}2\right)\G\left(\frac{5-p^r}4\right)}{\G\left(\frac{3-p^r}4\right)
\G\left(\frac{3-p^r}2\right)} .$$ The ratio on the left side of Theorem \ref{thm:1} then becomes a $\G$-quotient
times a power of $4$. The $\G$-quotient is converted to a $\G_p$-quotient.
We use the Kazandzidis supercongruence  to $p$-adically approximate this power of $4$
mod $p^{2r}$ and a Taylor expansion  approximates the $\G_p$-quotient. The coefficients of the first order
$p^r$ terms in the two approximations cancel, giving the desired supercongruence.

Compared to \cite{Long} we provide a
systematic way to carry out the $p$-adic analysis using the local analytic behavior of the $p$-adic Gamma function that handles  the ordinary and supersingular cases equally well. Our method also  explains the similarity between $p_{\{\frac12,\frac12:1\}}(-1)$ and the right  side of \eqref{eq:notsuper}. It is particularly effective when there are formulae for computing the complex periods as quotients of Gamma functions.  For instance, we are able to prove the following result
\begin{theorem}\label{thm:2}For $p\ge 5$ a prime the following congruence holds mod $p^6$:
\begin{equation}\label{eq:vHamme}
\pFq{7}{6}{\frac 76,\frac 13, \frac 13,\frac 13,\frac 13,\frac 13,\frac 13}{\frac 16,1,1,1,1,1}{1}_{{p-1}}=\sum_{k=0}^{{p-1}}(6k+1)
\frac{\left( \frac{1}{3} \right)^6_k}{k!^6}\equiv \left \{ \begin{array}{ll} -p\G_p\left(\frac13\right)^ 9 & \text{ if } p\equiv 1\mod 6\\
-\frac{10}{27}p^4\G_p\left(\frac13\right)^ 9 & \text{ if } p\equiv 5\mod 6. \end{array}\right .
\end{equation}
\end{theorem}
\noindent
This result is stronger than a prediction of van Hamme in \cite{vanHamme} which asserts a mod $p^4$ congruence for $p\equiv 1\mod 6$.
Computations yield that the only prime %, {   including both $p\equiv 1$ or $5 \mod 6$,}
between $5$ and $500$ satisfying the above congruence
mod $p^7$ is $19$.

 It is known that  for $\l\in \Q$ the  zeta-function of the $K3$ surface $X_\l$ can be computed
from a weight-3 Hecke eigenform $f_\l$  if and only if $\l$ is in
$\left\{\pm 1,4,\frac14, -8,-\frac18,64,\frac{1}{64}\right\},$ see \cite{SB} and \cite{AOP}.
For $\l$ in this set and almost all primes $p$ the following congruence holds
\begin{equation*}
 \pFq{3}{2}{\frac 12,\frac 12,\frac 12}{1,1}{\l}_{p-1}\equiv a_p(f_\l)\mod p^2
\end{equation*}
{   where $a_p(f_\l)$ is the $p$th coefficient of an explicit weight-3 Hecke eigenform $f_\l$ which
can be derived from \cite{KLMSY}. For instance, $f_1=\eta(4z)^6$ where $\eta(z)$ is the standard Dedekin eta function.}
We end the paper with a few supercongruences experimentally discovered by J. Kibelbek for ordinary primes and  stated in \cite{KLMSY}.  For instance, we prove
\begin{theorem}\label{thm:3} The following congruence holds modulo $p^3$
 $$\pFq{3}{2}{\frac 12,\frac 12,\frac 12}{,1,1}{1}_{p-1}\equiv \left \{ \begin{array}{ll} - \G_p\left(\frac14\right)^4 & \text{ if } p\equiv 1 \mod 4\\
 -\frac{p^2}{16}\G_p\left(\frac 14\right)^4  & \text{ if } p\equiv 3 \mod 4. \end{array} \right. $$

\end{theorem}
\noindent
The corresponding$\mod p^2$ congruence when $p\equiv 1\mod 4$ was first established by van Hamme \cite{vanHamme}.
  We obtain  a similar result for $-\frac 18$ at ordinary primes.
\medskip

We briefly outline the paper. Section \ref{ss:padic} includes the basic properties
of $p$-adic Gamma functions we need, including
Theorem~\ref{truncate} which we use repeatedly, and Lemma~\ref{GammaP}
which allows us to convert  $\G$-quotients to $\G_p$-quotients,  which  can then
be $p$-adically approximated by Theorem~\ref{truncate}. We give a simple application
of this in \S$2.2$. In \S\ref{ss:K-super} we use the Kazandzidis supercongruence
of binomial coefficients to prove Proposition~\ref{prop:G14G12}, a relation
between a $p$-adic logarithm and logarithmic derivatives of $\G_p$
that is used in the proof of Theorem~\ref{thm:1}. Section \ref{ss:vanhamme} is devoted
to the proof of Theorem~\ref{thm:2} and a corollary.
In \S \ref{ss:other} we prove Theorem~\ref{thm:3} and another result of similar type.
The appendix, section \ref{ss:Appendix}, includes the hypergeometric transformation and evaluation formulae
we use and includes detailed statements of their domains of convergence.
Many of the theorems here were motivated by computations carried out using the open source software \texttt{Sage}.
We thank Heng Huat Chan for conversations which led to \S2.2, Wadim Zudilin for inspiring discussions  and the reference \cite{Slate}  and Jonas Kibelbek for his comments on an earlier version.

\section{$p$-adic Gamma functions and immediate applications}\label{ss:padic}

\subsection{Preliminaries}\label{ss:Preliminaries}
Throughout this paper {    $p\ge5$} is a prime, $v_p(\cdot)$ denotes the $p$-order and $|x|_p=p^{-v_p(x)}$ the $p$-adic norm.
The main results of this section that are used later in the paper are
Theorem~\ref{truncate} and Lemma~\ref{GammaP}. If the reader wishes, she may
assume these, noting that $G_k(a):= \Gamma^{(k)}_p(a) /\Gamma_p(a)$, and proceed to \S\ref{ss:K-super}.

For the sake of completeness,
we recall some basic properties of the Morita $p$-adic
Gamma function $\G_p(x)$ for $x\in \Z_p$.
None of the results of this subsection are new, but we gather them here for our convenience and
hopefully that of the reader.
For more details, see \cite{M} and \cite[\S 11.5]{Cohen}.
As an immediate application we reprove and generalize a result from \cite{CKKO}.

\begin{propdef}\label{pGamma}
\begin{itemize}
\item[1).] $\G_p(0)=1$
\item[2).] $\displaystyle \frac{\G_p(x+1)}{\G_p(x)}=\left \{  \begin{array}{lll} -x & \text{if } |x|_p=1\\
-1& \text{if } |x|_p<1. \end{array}\right .$
\item[3).] $\G_p(x)\G_p(1-x)=(-1)^{a_0(x)}$ where
$a_0(x)\in \{1,2,\cdots, p\}$ satisfies $x-a_0(x)\equiv 0 \mod p$.
\item[4).] $\displaystyle\G_p\left(\frac12\right)^2=(-1)^{\frac{p+1}2}$.
\end{itemize}
\end{propdef}

We use Theorem~\ref{truncate} repeatedly in this paper.
Its proof involves essentially reproving:

\begin{theorem}[Morita, Barsky] For $a\in \Z_p$ the function
$x \mapsto \G_p(a+x)$ is locally analytic on
$\Z_p$  and converges for $v_p(x) \geq \displaystyle \frac1{p}+\frac1{p-1}$.
\end{theorem}

Recall the $p$-adic logarithm
$$\log_p(1+x)=\sum_{n=1}^\infty \frac{(-1)^{n+1}x^n}{n},$$ which converges for $x\in \mathbb C_p$ with $|x|_p<1$.

Set $G_k(a) = \G^{(k)}_p(a)/\G_p(a)$. In particular, $G_0(a)=1$.
\begin{cor} \label{reflect}  For a in $\Z_p$, $G_1(a)=G_1(1-a)$,
and $G_2(a)+G_2(1-a) =2G^2_1(a)$.
\end{cor}
\begin{proof}
Take $\log_p$ of 3) of Proposition/Definition~\ref{pGamma} and, on
observing that $a_0(x)$ is constant in a small enough neighborhood of $x$,
we differentiate this
to obtain
$G_1(x)=G_1(1-x)$. Rewriting this in terms of $p$-adic Gamma functions and differentiating again
gives $G_2(x)-G^2_1(x) = -G_2(1-x) +G^2_1(1-x)$. The second part follows from the first part.
\end{proof}

\begin{theorem}\label{RZ}[Robert-Zuber, \cite{RZ95}]
Assume $p\ge 5$ is a prime. The function $\log_p \G_p(x)$ is an odd  analytic
function on $p\Z_p$ such that
\begin{equation}\label{eq:RZ1}
\log_p \G_p(x)=\l_0 x -\sum_{m\ge 1}\frac{\l_m}{2m(2m+1)}\cdot x^{2m+1}
\end{equation}
where $\l_1\in \Z_p$
and $p\l_m\in \Z_p$ when $m\ge 2$.
\end{theorem}

\begin{cor}\label{G2G1}
$G_2(0)= G^2_1(0)$.
\end{cor}
\begin{proof} Differentiate \eqref{eq:RZ1} twice and plug in $x=0$.% and use Corollary~\ref{reflect}.
\end{proof}

\begin{lemma}\label{derivs}
The $k$th derivative of $\log \G_p(x)$ is a polynomial
with integer coefficients in the $G_i(x)$ for $i\leq k$ and the
sum of the subscripts in each monomial is $k$. The coefficient
of $G_k(x)$ is $1$.
\end{lemma}
\begin{proof} We induct on $k$, the case $k=1$ being trivial.
Suppose the statement is true for $k-1$, that is the $k-1$st derivative of
$\log \G_p(x)$ is of the form
$\displaystyle  \sum_l a_{r_l}\prod^{r_l}_{j=1}
G_{i_j}(x) $ where $a_{r_l} \in {\mathbb Z}$, the subscripts
in each monomial add to $k-1$
and the term $\displaystyle G_{k-1}(x) $ appears with coefficient
$1$.
The derivative of $G_{k-1}(x)$ is
$$ \displaystyle \frac{ \G_p^{(k)}(x) \G_p(x) - \G_p^{(k-1)}(x) \G_p^{(1)}(x)  } {\G_p(x)^2}
= G_k(x) -G_{k-1}(x)G_1(x).$$ Since no other term in the sum includes a $k-1$st derivative,
this completes the induction for  the coefficient
of $G_k(x)$. Completing the rest of the induction involves writing
the monomial $ \prod^{r_l}_{j=1} G_{i_j}(x) $
as $\displaystyle\frac{ \G_p^{(i_1)}(x)  \G_p^{(i_2)}(x) \cdots \G_p^{(i_{r_l})}(x)}{ (\G_p(x))^{r_l}}$
and employing the quotient and product rules.
\end{proof}
\begin{lemma} \label{factorial}$\displaystyle v_p\left(\frac{1}{k!}\right) \geq -\frac{k}{p-1}$.
\end{lemma}
\begin{proof}
$\displaystyle  v_p(k!) =
\sum^{\infty}_{r=1} \left[\frac{k}{p^r}\right] \leq \sum^{\infty}_{r=1} \frac{k}{p^r} =
\frac{k}{p-1}.$
\end{proof}

\begin{lemma}\label{valGi}
Let  $p\geq 5$. Then $v_p(G_1(0)) \geq 0$. Also,
$\displaystyle v_p(G_i(0)) \geq -\left[\frac{i}{p}\right]$ for $i>1$. In particular, $v_p(G_i(0)) \geq 0$ for
$i<p$
and $v_p(G_{p}(0))=-1$.
\end{lemma}
\begin{proof}
It may be useful to the reader to note here that
$G_{i}(0)$ is related to the $i$th derivative of the expression
for $\log_p \G_p$ below, and is thus given in terms of
$\l_{ \frac{i-1}{2}}$.
From Theorem~\ref{RZ} we know that
$$\displaystyle \log \Gamma_p(x)= \lambda_0x -\sum_{m\geq 1} \frac{\lambda_m}{2m(2m+1)}\cdot
x^{2m+1}$$
is analytic for $x \in p{\mathbb Z}_p$.

Note $\lambda_0 = G_1(0) =\displaystyle \frac{\Gamma^{'}_p(0)}{\Gamma_p(0)} = \Gamma^{'}_p(0)
=\lim_{r \to \infty} \frac{ \Gamma_p(p^r)-1} {p^r}$.
Now $\Gamma_p(p^r) = (-1)^{p^r}\prod^{p^r}_{k=1, (k,p)=1} k$
is by Wilson's theorem $(-1)^{p^r}(-1)\equiv 1$ mod $p^r$
so $v_p(\Gamma^{'}_p(0))=v_p(G_1(0))\geq0$. (D. Thakur has pointed out to us that the only known
cases where this valuation is positive are  $p=5,13,563$ and in these cases $v_p(G_1(0))=1$.)

For $m\geq 1$
Lemma $1$ of \cite{RZ95}  gives $\lambda_m \equiv B_{2m}$
mod ${\mathbb Z}_p$ where $B_{2m}$ is the $2m$th Bernoulli number.
Take $2m$ derivatives of the  expression for $\log \Gamma_p(x)$, plug in
$x=0$
and use Lemma~\ref{derivs} to get that $G_{2m}(0)$ is a polynomial with integer coefficients
in the terms $G_k(0)$ for $k<2m$. Similarly,
$G_{2m+1}(0)$ is the difference between such a polynomial with $k<2m+1$
and
$\displaystyle \frac{(2m+1)!\lambda_{m}}{2m(2m+1)} = (2m-1)!\lambda_m$.
The Clausen-von Staudt Theorem
states that the denominator of $B_{2m}$ is $\displaystyle \prod_{\ell-1 | 2m}  \ell$.
So $v_p(G_i(0)) =0$ for $i<p$.
The denominator of $B_{p-1}$  is exactly divisible by $p$
 so $v_p(G_p(0)) =-1$. For   $p<i<2p$, $G_i(0)$ is the sum of
a polynomial
in the previous $G_j(0)$ (with $G_p(0)$ occurring
to at most first degree in any monomial by Lemma~\ref{derivs})  and
$(i-2)!\lambda_{\frac{i-1}{2}}$ which is integral in $p$. The formula for $G_{2p}(0)$
contains the term $G^2_p(0)$ which has $p$-adic valuation $-2$,
so $v_p(G_{2p}(0)) \geq -2$. Continuing in this fashion gives
$v_p(G_i(0)) \geq \displaystyle -\left[\frac{i}{p}\right]$.
\end{proof}
We immediately see:
\begin{cor}\label{vpGi} For $i<p$,
$\displaystyle
v_p\left(\frac{\G^{(i)}_p(0)}{i!}\right)=
v_p\left(\frac{G_i(0)}{i!}\right)=0$.
For all $i$, $\displaystyle
v_p\left(\frac{\G^{(i)}_p(0)}{i!}\right)=
v_p\left(\frac{G_i(0)}{i!}\right) \geq
-i\left(\frac{1}{p}+\frac{1}{p-1}\right)$.
For $x \in {\mathbb C}_p$ satisfying $v_p(x) >
\displaystyle \left(\frac{1}{p}+\frac{1}{p-1}\right)$, the Taylor series
$\G_p(x) = \displaystyle \sum^{\infty}_{k=0} \frac{ \G^{(k)}_p(0) }{k!} x^k$ converges.
Dividing by $\G_p(0)$ this becomes
$ \displaystyle \frac{\G_p(0+x)}{\G_p(0)} =
\sum^{\infty}_{k=0} \frac{ G_k(0) }{k!} x^k$.
\end{cor}

We `transfer' this result to arbitrary $a \in \Q \cap \Z_p$.
As the proof of the Proposition~\ref{prop:valLB} is a routine exercise in
expansions of nonarchimedean series about different points, we do not include it.

\begin{prop}\label{prop:valLB}
Let $p\geq 5$ and  $a\in {\mathbb Q}$ with $v_p(a) \geq 0$. Then
$ \displaystyle  v_p\left(\frac{G_i(a)}{i!}\right) \geq
-i\left(\frac{1}{p}+\frac{1}{p-1}\right)$.
For $i<p$, $ \displaystyle  v_p\left(\frac{G_i(a)}{i!}\right)=0$.
We may extend the domain of $\G_p(a+x)$ by setting
$ \displaystyle {\G_p(a+x)}= {\G_p(a)} \cdot
\sum^{\infty}_{k=0} \frac{ G_k(a) }{k!} x^k$
for $x\in {\mathbb C}_p$ with
$ \displaystyle  v_p(x)\geq \left(\frac{1}{p}+\frac{1}{p-1}\right)$.
In particular,
\begin{equation}
\label{eq:recursion}
\displaystyle\frac{\G_p(a+1+x)}{\G_p(a+x)}=\left\{\begin{array}{cc} -(a+x) & |a+x|_p=1
\\ -1& |a+x|_p<1\end{array}\right.  ,
\end{equation} and
\begin{equation}
\label{eq:functional}
\G_p(a+x)\G_p(1-a-x)=(-1)^{a_0(a)}.
\end{equation}
\end{prop}

Theorem \ref{truncate} provides a $p$-adic approximation to $\G_p$-quotients.

\begin{theorem}\label{truncate}
For $p \geq 5$, $r \in {\mathbb N}$,
{$a\in \Z_p ,m\in{\mathbb C_p}$ satisfying $ \displaystyle  v_p(m)\geq 0$}
and $t\in \{0,1,2\}$
we have
$$\frac{ \G_p(a+mp^r)}{\G_p(a)}\equiv \sum^t_{k=0} \frac{ G_k(a)}{k!} (mp^r)^k \mod p^{(t+1)r}.$$
The above result also holds for $t=4$ if $p \geq 11$.
\end{theorem}

\begin{proof}
When $t=1$  we need only show that all terms past the first degree term
in the Taylor series have valuation at least $2r$. Since $p \geq 5$,
Proposition~\ref{prop:valLB} implies  $\displaystyle \frac{G_k(a)}{k!} \in \Z_p$
for $k=2,3,4$
so $\displaystyle v_p\left(\frac{G_k(a)}{k!}(mp^r)^k\right) =kr \geq 2r$ so
there is no problem with the $2$nd, $3$rd and $4$th terms.
For $k \geq 5$ we have
$\displaystyle v_p\left(\frac{G_k(a)}{k!} (mp^r)^k\right)
\geq kr-k\left(\frac{1}{p} +\frac{1}{p-1}\right)$. We
need to check  this is
at least $2r$.
A simple computation reduces this to the inequality
$$r \geq \left(\frac{k}{k-2}\right)\left(\frac{1}{p} +\frac{1}{p-1}\right).$$
For $k,p \geq 5$ the right side is less than $1$
so the inequality holds for all $r\in {\mathbb N}$.

When $t=2$,  the $3$rd and $4$th  terms are  handled as above.
For $k \geq 5$
 the relevant inequality
is $$r \geq \left(\frac{k}{k-3}\right)\left(\frac{1}{p} +\frac{1}{p-1}\right).$$
The right side is less than $1$ when $k \geq 5$ and $p \geq 7$.
It is also less than $1$ when $k \geq 6$ and $p=5$. It remains to check the $k=p=5$ case,
namely that
$\displaystyle v_p \left(\frac{G_5(a)}{5!} (m5^r)^5\right)
\geq 3r$. Proposition~\ref{prop:valLB} gives this valuation is $5r-2$, which is at least
$3r$ for all natural numbers $r$.

When $t=4,\,p \geq 11$, $\displaystyle v_p\left(  \frac{G_k(a)}{k!} (mp^r)^k\right)
\geq kr$ for $5 \leq k \leq 10$  by arguments similar to those above. Again, arguing as above for $k \geq 11$,
we need  $$r \geq \left(\frac{k}{k-5}\right)\left(\frac{1}{p} +\frac{1}{p-1}\right).$$ As the right side is less than $1$
when $k,p \geq 11 $, we are done.

We leave the $t=0$ case as
an exercise.

\end{proof}

We will need some formulae for various $\G$ and $\G_p$-quotients.
\begin{defn}\label{def:a'} Let $a \in \Q$ with $v_p(a)=0$. Let
$i$ be the unique integer in $\{1,2,\dots,p-1\}$
satisfying $a+i \equiv 0 \mod p$ and define $a' \in \Q$
by $a+i=pa'$.
\end{defn}

\begin{lemma}\label{aprime}
For $\displaystyle
a \in \left\{\frac1n,\frac2n,\cdots,\frac{n-1}n\right\}$ and $p \equiv 1 \mod  n$, $a'=a$.
\end{lemma}
\begin{proof}
%1)
$\displaystyle \frac{k}n+ k\cdot\frac{p-1}n = \frac{kp}n$.
\end{proof}

Lemma \ref{GammaP} below is the technical ingredient that allows us to replace
$\G$-quotients with $\G_p$-quotients.

\begin{lemma}\label{GammaP}
Let $a \in (0,1] \cap {\mathbb Q}$.
\begin{itemize}
\item[1)] If $v_p(a) =0$
then $\forall m,r \in {\mathbb N}$,
$$\displaystyle \frac{ \Gamma(a+mp^r)}{  \Gamma(a+mp^{r-1})} =
(-1)^m p^{mp^{r-1}} \frac{ \Gamma_p(a+mp^r)}{  \Gamma_p(a)}
\frac{(a')_{mp^{r-1}}} {(a)_{mp^{r-1}}}.$$
\item[2)] Suppose $a+mp^r \in \N \,\,\forall r \in \N$.
(Here, $a,m\in \Q$ but need {\bf not} be in $\Z$.)
Then
$$\displaystyle \frac{ \Gamma(a+mp^r)}{  \Gamma(a+mp^{r-1})} =
(-1)^{a+mp^r} p^{a+mp^{r-1}-1} \Gamma_p(a+mp^r).$$
\item[3)] Let $a,b \in {\mathbb Q}$ and suppose $a-b \in \Z$ and $a,b \notin \Z_{\leq 0}$.
If none of the numbers
between $a$ and $b$ that differ from both by an integer are divisible by $p$
then
$\displaystyle \frac{ \G(a)}{\G(b)} = (-1)^{a-b}
\frac{\G_p(a)}{\G_p(b)}$. Equivalently, {    if $a\in \Z_p, n\in \mathbb N$ such that none of $a,a+1\cdots, a+n-1$ is $p\Z_p$, then $$(a)_n=(-1)^n \frac{\G_p(a+n)}{\G_p(a)}.$$}
\newline\noindent
Variant: If there are numbers $x$  between $a$ and $b$
with $x-a,x-b \in \Z$ and $v_p(x)>0$, then the right side must include these multiples.
\end{itemize}
\end{lemma}
\begin{proof}

1)  As $a'+mp^{r-1}>0$, $\G(a'+mp^{r-1}) \neq 0$ so
$$ \frac{ \Gamma(a+mp^r)}{  \Gamma(a+mp^{r-1})} =
\frac{ \Gamma(a+mp^r)}{  \Gamma(a'+mp^{r-1})}\frac{ \Gamma(a'+mp^{r-1})}{  \Gamma(a+mp^{r-1})}$$
which, by the fundamental property of the $\G$-function equals
$$
\frac{\Gamma(a) \cdot a\cdot (a+1) \cdots (a+mp^r-1) }
{\Gamma(a') \cdot a' \cdot (a'+1) \cdots (a'+mp^{r-1}-1)}
\cdot \frac{ \Gamma(a'+mp^{r-1})}{  \Gamma(a+mp^{r-1})}.$$
Now $a+i$ is the first term in the numerator that is a multiple of $p$
and this is $pa'$. Similarly $a+i+p =p(a'+1),\dots, a+i+mp^r-p = p(a'+mp^{r-1}-1)$.
The multiples of $p$ in the numerator cancel exactly
with the terms in the denominator
leaving
$$\displaystyle \frac{\G(a)} {\G(a')} \left(p^{mp^{r-1}} \tilde{\prod}^{mp^r-1}_{j=0} (a+j) \right)
\frac{ \Gamma(a'+mp^{r-1})}{  \Gamma(a+mp^{r-1})}
=  \frac{\G(a)} {\G(a')} \left(
p^{mp^{r-1}}
\frac{\G_p(a)}{\G_p(a)}
\tilde{\prod}^{mp^r-1}_{j=0} (a+j) \right)
\frac{ \Gamma(a'+mp^{r-1})}{  \Gamma(a+mp^{r-1})}$$
where the products are over terms prime to $p$.
By the definition of $\G_p$  this is
$$\displaystyle  \frac{\G(a)} {\G(a')} p^{mp^{r-1}} (-1)^{mp^r}
\frac{\Gamma_p(a+mp^r)} {\Gamma_p(a)} \frac{\G(a'+mp^{r-1})}{\G(a+mp^{r-1})}=
 (-1)^mp^{mp^{r-1}}
\frac{\Gamma_p(a+mp^r)} {\Gamma_p(a)} \frac{ (a')_{mp^{r-1}}}{(a)_{mp^{r-1}}}   ,$$
where the last equality follows from using the definition of rising
factorials twice.

2) This is like part 1), except here
$$ \frac{ \Gamma(a+mp^r)}{  \Gamma(a+mp^{r-1})} =
\frac{ (a+mp^r-1)!}{(a+mp^{r-1}-1 )!}.$$
It is easy to see, using that $a \in (0,1] \cap \Q$, that each  multiple of $p$ in the numerator is $p$ times
an element of the denominator. Canceling these, introducing $1$ in the form
$\displaystyle \frac{\G_p(1)} {\G_p(1)} $ and using
the basic properties of $\G_p$, including that $\G_p(1)=-1$,  gives the result.

3) This follows immediately from the definitions of $\G$ and $\G_p$.
\end{proof}

\subsection{ An immediate application}
Lemma $2.1$ of  \cite{CKKO} asserts that for each prime $p\equiv 1\mod 4$
\begin{equation}
\left (\frac34 \right )_p\equiv 3\left (\frac14 \right )_p \mod p^3.
\end{equation}
We generalize this as follows.
{   \begin{lemma} \label{generalizeCKKO} Let $a\in (0,1) \cap \Q$ with $v_p(a)=0$. For $p \geq 5$
\begin{equation}
\frac{(a) _{p^r}}{(a')_{p^{r-1}}} \equiv \frac{(1-a)_{p^r}}{(1-a')_{p^{r-1}}}  \mod p^{p^{r-1}+2r}.
\end{equation}
\end{lemma}
\begin{proof}
By definition $(a)_{p^r} =a(a+1)\cdots(a+p^r-1)$. The set $\{a+i\}_{i=0}^{p^r-1}$  contains $p^{r-1}$ multiples of $p$, which are
$a'p, (a'+1)p, \cdots, (a'+p^{r-1}-1)p$ respectively, combining with $ \displaystyle (a)_{p^r} =\frac{\G(a+p^r)}{\G(a)}$ we have
$ \frac{(a) _{p^r}}{(a')_{p^{r-1}}} =p^{p^{r-1}}(-1)^{p^r}\frac{\G_p(a+p^r)}{\G_p(a)}$. By Theorem \ref{truncate}, we have
$$ \frac{(a) _{p^r}}{(a')_{p^{r-1}}} =p^{p^{r-1}}(-1)^{p^r}(1+G_1(a)p^r+G_2(a)p^{2r}) \mod p^{p^{r-1}+2r}.$$
Similarly, $$ \frac{(1-a) _{p^r}}{(1-a')_{p^{r-1}}} =p^{p^{r-1}}(-1)^{p^r}(1+G_1(1-a)p^r+G_2(1-a)p^{2r}) \mod p^{p^{r-1}+2r}.$$ The fact that
$G_1(a)=G_1(1-a)$  gives the result.
\end{proof}}
\begin{cor}
Fix $n \in {\mathbb N}$. For  any prime $p \equiv 1 \mod n$,
\begin{equation}
\left(1-\frac1n \right)_p\equiv (n-1)\left (\frac1n \right )_p \mod p^3.
\end{equation}
\end{cor}
\begin{proof} Set $\displaystyle a=1-\frac1n$ in Lemma~\ref{generalizeCKKO} and use
Lemma~\ref{aprime}.
Setting $n=4$ recovers the result of \cite{CKKO}.
\end{proof}

\section{Kazandzidis supercongruences and the proof of Theorem \ref{thm:1}}\label{ss:K-super}
We prove Theorem \ref{thm:1}.
We first use the Kazandzidis supercongruence for binomial
coefficients to $p$-adically approximate $4^{\frac{p^r-p^{r-1}}4}$ mod $p^{2r}$ in Proposition
\ref{prop:G14G12}. The
{\it Pfaff transformation formula} allows us to rewrite our $_2F_1$ as another such expression and
then   {\it Kummer's evaluation
formula} converts this $_2F_1$ to a  $\G$-quotient. Thus the left side of Theorem~\ref{thm:1} becomes
a  $\G$-quotient mutiplied by a power of $4^{\frac{p^r-p^{r-1}}4}$, which we already approximated.
Lemma \ref{GammaP} expresses the $\G$-quotient  as a $\G_p$-quotient and then we use
Theorem~\ref{truncate} and the result for $4^{\frac{p^r-p^{r-1}}4}$ to approximate it.

\subsection{ Approximating $4^{\frac{ p^r-p^{r-1}}4}$.}

In 1968, Kazandzidis proved that for $0\le m\le n$ and prime  $p\ge 5$,
$$\binom{pn}{pm}\equiv \binom{n}{m} \mod p^3.$$  When $n=2$ and $m=1$, it is equivalent to the Wolstenholme's theorem.
 In \cite{RZ95}, Robert and Zuber generalized this  to
\begin{equation}\label{eq:RZ}
\binom{p^rn}{p^rm}\equiv \binom{p^{r-1}n}{p^{r-1}m} \mod p^{3r}.
\end{equation}
In fact, they proved a stronger congruence, but the above is all we need.
The expression
$\binom{p^rn}{p^rm}/ \binom{p^{r-1}n}{p^{r-1}m} $ can be written in terms of $p$-adic Gamma functions.

We need some easy lemmas whose proofs we omit.
\begin{lemma}\label{binomunit} For $p \geq 5$,
$\displaystyle v_p\left(\binom{2mp^r}{mp^r}\right)=0$ for $m=1,2$, $r\ge 0$.
\end{lemma}
\begin{lemma} \label{specialtruncate}
Let $p\geq 5$ and $x \in p^r\Z_p$. Then $\log_p(1+x) \equiv x \mod p^{2r}$,
$\displaystyle\log_p(1+x) \equiv x -\frac{x^2}2 \mod p^{3r}$,
$e^x \equiv 1+x \mod p^{2r}$ and $\displaystyle e^x \equiv 1+x +\frac{x^2}{2!}\mod p^{3r}$.
\end{lemma}

We also need the Legendre duplication formula,
\begin{equation} \label{eq:Legendre} \G(2z) = \frac{2^{2z-1}\G(z)\G(z+1/2) }{\sqrt{\pi}}.
\end{equation}
By the fundamental property of the Gamma function
$$\binom{2p^r}{p^r}/\binom{2p^{r-1}}{p^{r-1}}=
\frac{ \G(1+2p^r)}{ (\G(1+p^r))^2} \frac{ (\G(1+p^{r-1}))^2}{\G(1+2p^{r-1})}.$$
Taking
$2z=1+2p^r$ and $2z=1+2p^{r-1}$ in (\ref{eq:Legendre})
\begin{multline*}\binom{2p^r}{p^r}/\binom{2p^{r-1}}{p^{r-1}}=
\frac{2^{2p^r}\G(\frac{1}{2}+p^r)\G(1+p^r) }{ \sqrt{\pi} (\G(1+p^r))^2}
\frac{ \sqrt{\pi} (\G(1+p^{r-1}))^2}{ 2^{2p^{r-1}} \G(\frac{1}{2}+p^{r-1}) \G(1+p^{r-1})}\\
= 2^{2(p^r-p^{r-1})} \frac{\G( \frac{1}{2}+p^r) } {\G( \frac{1}{2}+p^{r-1}) }
\frac{\G(1+p^{r-1})}{\G(1+p^r)} =  4^{(p^r-p^{r-1})}
\frac{ \G_p(\frac{1}{2} +p^r)}{\G_p(\frac{1}{2}) } \frac{\G_p(1)}{\G_p(1+p^r)}
,\end{multline*} where the last equality follows from two applications of 1) of Lemma~\ref{GammaP}. The various powers of $(-1)$
cancel out.

Theorem~\ref{truncate} implies that taking Taylor expansions of these $\G_p$'s and truncating mod $p^{2r}$ gives congruences
mod $p^{2r}$. Then taking $\log_p$ and truncating mod $p^{2r}$ perpetuates
these congruences by Lemma~\ref{specialtruncate}.
Using~\eqref{eq:RZ}  and invoking Lemma~\ref{binomunit} to divide yields
\begin{multline}\label{eq:Gvalues}
p^{r-1}\log_p 4^{p-1}+\log_p\left[1+G_1\left(\frac12\right)p^r+G_2\left(\frac12\right)
\frac{(p^{r})^2}2+\cdots\right]-\\
\log_p\left[1+G_1(0)p^r+G_2(0)\frac{(p^r)^2}2+\cdots  \right]  \equiv 0\mod p^{3r}.
\end{multline}
We see
$$p^r \frac{1}{p} \log_p 4^{p-1} +G_1\left(\frac12\right)p^r -G_1(0)p^r \equiv 0 \mod p^{2r}$$
holds $\forall r$. Letting $r \to \infty$ gives
\begin{equation}\label{eq:log41}
\frac1p\log_p4^{p-1} +G_1\left(\frac12\right) -G_1(0)=0.
\end{equation}
The equation below is derived similarly from the
the Kazandzidis supercongruence. We go through two iterations of Legendre's duplication formula
\eqref{eq:Legendre} and four applications
of 1) of Lemma~\ref{GammaP} to get
$$1\equiv \binom{4p^r}{2p^r}/\binom{4p^{r-1}}{2p^{r-1}}=2^{4(p^r-p^{r-1})}
\frac{\G_p\left(\frac14+p^r\right)\G_p\left(\frac34+p^r\right)
\G_p\left(\frac12\right)\G_p(1)}{\G_p(1+p^r)\G_p\left( \frac12+p^r\right)
\G_p\left(\frac14\right)\G_p\left(\frac34\right)}
\mod p^{2r}.$$
Note that $\G_p\left(\frac14\right)\G_p\left(\frac34\right)=(-1)^{a_0(1/4)}$, $\G_p(1)=-1$ and
$\G_p\left(\frac12\right)$ is some $4$th root of unity. Thus $\log_p$ of the constant part of the expression
above is $0$.
Taking $\log_p$ and using Theorem~\ref{truncate} to
reduce mod $p^{2r}$  we see
\begin{equation}\label{eq:log42}
\frac2p \log_p(4^{p-1})+G_1\left(\frac14\right)+G_1\left(\frac34\right) -G_1(1)-
G_1\left(\frac12\right)=0.
\end{equation}
Subtracting (\ref{eq:log41}) from (\ref{eq:log42}) and using Corollary~\ref{reflect}
gives
\begin{prop}\label{prop:G14G12} $\displaystyle \frac1p \log_p 4^{p-1} =2G_1\left(\frac12\right)-2G_1
\left(\frac14\right)$. Multiplying by $\displaystyle \frac{p^r}4$ and exponentiating, this becomes
$$\displaystyle 4^{\frac{p^r-p^{r-1}}4}
\equiv
(-1)^{\frac{p^2-1}8} \left [ 1 + \left( G_1\left(\frac12\right)-G_1\left(\frac14\right)\right)\frac{p^r}2 \right ] \mod p^{2r} .$${    The sign  $(-1)^{\frac{p^2-1}8}$ is due to the Legendre symbol
$\left (\frac{2}p \right )=(-1)^{\frac{p^2-1}8}$.}% when $p\equiv 1\mod 8$ and it is $-1$ when $p\equiv 5 \mod 8$.
\end{prop}

We need \eqref{eq:4power2} below for use in \eqref{eq:for35} and \eqref{eq:5.15}. We only sketch its derivation.
It requires two applications of the duplication formula, Euler's reflection formula ~\eqref{eq:Eulerrefl}
and $2$) and the variant of $3$) of Lemma~\ref{GammaP}.
  Multiplying the congruences
$\displaystyle\binom{2p}{p}/\binom{2}{1}\equiv 1\mod p^3$ and $\displaystyle\binom{4p}{2p}
/\binom{4}{2}
\equiv 1\mod p^3$ we have
$$64^{p-1}\frac{\G_p\left(\frac14+p\right)\G_p\left(\frac34+p\right)
}{\G_p(1+p)^2\G_p(\frac 14)\G_p(\frac 34)} \equiv 1\mod p^3.$$ Thus
\begin{equation*}
64^{(p-1)}\equiv  1+\left(2G_1(0)-2G_1\left(\frac 14\right)\right)p +
2\left( G_1(0)-G_1\left(\frac 14\right)\right)^2 p^2  \mod p^3
\end{equation*}It follows again from  $\left ( \frac 2p\right )=(-1)^{\frac{p^2-1}8}$
and Lemma~\ref{specialtruncate}
that
\begin{equation}\label{eq:4power2}
2^{\frac{p-1}2}\equiv  (-1)^{\frac{p^2-1}8}
\left[ 1+\frac16\left(G_1(0)-G_1\left(\frac 14\right)\right)p +
\frac 1{72} \left( G_1(0)-G_1\left(\frac 14\right)\right)^2p^2 \right ] \mod p^3
\end{equation}
and cubing this we have
\begin{equation}\label{eq:8power2}
2^{\frac{3(p-1)}2}\equiv  (-1)^{\frac{p^2-1}8}
\left[ 1+{\frac12}\left(G_1(0)-G_1\left(\frac 14\right)\right)p +
\frac 18 \left( G_1(0)-G_1\left(\frac 14\right)\right)^2p^2 \right ] \mod p^3
\end{equation}

\subsection{Approximating the $_2F_1$ ratio and the Proof of Theorem \ref{thm:1}}
We are ready to prove Theorem \ref{thm:1}. We will need the following
{\it Pfaff transform} (see  \cite[eq. (2.3.14), pp. 79]{AAR})
\begin{equation}\label{eq:pfaff1}
\pFq{2}{1}{-n,b}{c}{x}=\frac{(c-b)_n}{(c)_n} \, \pFq{2}{1}{-n,b}{b+1-n-c}{1-x}.
\end{equation}
as well as \emph{Kummer's evaluation formula}, see (2.11) of \cite{Whipple}
\begin{equation}\label{eq:Kummer1}
\pFq{2}{1}{a,b}{1+a-b}{-1}=\frac{\G(1+a-b)\G\left(1+\frac a2\right)}{\G\left(1+\frac a2-b\right)\G(1+a)}
\end{equation}

\begin{proof}[Proof of Theorem~\ref{thm:1}] Recall $\G\left(\frac12\right)=\sqrt{\pi}$
and Euler's reflection formula\begin{equation}\label{eq:Eulerrefl}
\G(z)\G(1-z) =\displaystyle \frac{\pi}{\sin (\pi z) }.
\end{equation}
We   let $x=2$ and $\displaystyle n=\frac{p^r-1}2$, $\displaystyle b=\frac 12,c=1$ in \eqref{eq:pfaff}.
\begin{eqnarray*}
\pFq{2}{1}{\frac{1-p^r}2,\frac {1}2}{1}{2}&\overset{\text{(by  \eqref{eq:pfaff1})}}{=}&
\frac{\left(\frac12\right)_{\frac{p^r-1}2}}{(1)_{\frac{p^r-1}2}}\,
\pFq{2}{1}{\frac{1-p^r}2, \frac12}{1-\frac{p^r}2}{-1}\\&
\overset{\text{(by \eqref{eq:Kummer1})}}{=} &\frac{(\frac12)_{\frac{p^r-1}2}}{(1)_{\frac{p^r-1}2}}\cdot
\frac{\G\left(1-\frac{p^r}2\right)\G\left(\frac{5-p^r}4\right)}{\G\left(\frac{3-p^r}4\right)
\G\left(\frac{3-p^r}2\right)} \\
&=& \frac{\G\left(\frac{p^r}2\right)/\G\left(\frac12\right)}{\G\left(\frac{1+p^r}2\right)/\G(1)}\cdot
\frac{\G\left(1-\frac{p^r}2\right)\G\left(\frac{5-p^r}4\right)}{\G\left(\frac{3-p^r}4\right)\G\left(\frac{3-p^r}2\right)}
\quad \text{(definition of rising factorials)}\\
&=&\frac{\G\left(\frac{p^r}2\right)\G(1)}{\G\left(\frac{1+p^r}2\right)
\G\left(\frac12\right)}\cdot
\frac{2^{\frac{p^r-1}2}\sqrt{\pi}\G\left(1-\frac{p^r}2\right)}{\G\left(\frac{3-p^r}4\right)^2}
\quad\text{ (set $z= (3-p^r)/4$ in \eqref{eq:Legendre})}\\
&=&
\frac{\G\left(\frac{p^r}2\right)}{\G\left(\frac{1+p^r}2\right)}\cdot
\frac{2^{\frac{p^r-1}2}\G\left(1-\frac{p^r}2\right)}{\G\left(\frac{3-p^r}4\right)^2}\\&
\overset{\text{ (by~\eqref{eq:Eulerrefl})}}{=} &\frac{2^{\frac{p^r-1}2}\pi}{\G\left(\frac{1+p^r}2\right)\G\left(\frac{3-p^r}4\right)^2} \quad .
\end{eqnarray*}

Recall $p\equiv 1\mod 4$ so $\G_p\left(\frac 12\right)^2=-1$. Then
\begin{eqnarray*}
\frac{\pFq{2}{1}{\frac{1-p^r}2\,\,\frac {1}2}{1}{2}}{\pFq{2}{1}{\frac{1-p^{r-1}}2\,\,\frac {1}2}{1}{2}}
&=&
4^{\frac{p^r-p^{r-1}}4}\frac{\G\left(\frac{1+p^{r-1}}2\right)\G\left(\frac{3-p^{r-1}}4\right)^2}
{\G\left(\frac{ 1+p^r}2\right)\G\left(\frac{3-p^r}4\right)^2}\\
&\overset{\text{ (by~\eqref{eq:Eulerrefl})}}{=}&4^{\frac{p^r-p^{r-1}}4}\frac{\G\left(\frac{1+p^{r-1}}2\right)\G\left(\frac{1+p^{r}}4\right)^2}
{\G\left(\frac{ 1+p^r}2\right)\G\left(\frac{1+p^{r-1}}4\right)^2} \quad \\
&\overset{}{=}&  4^{\frac{p^r-p^{r-1}}4}
\frac{  \left[ \frac{ (-1)^{\frac{1+p^r}4-\frac12} p^{\frac{1+p^{r-1}}4 -\frac12} \G_p\left(\frac{1+p^r}4\right)  }{\G_p\left(\frac12\right)}\right]^2} {(-1)^{\frac{1+p^r}2}p^{\frac{1+p^{r-1}}2 -1}\G_p\left(\frac{1+p^r}2\right)}\quad \text{ ( $2$) of Lemma~\ref{GammaP})}\\
&=&-4^{\frac{p^r-p^{r-1}}4} \frac{  \G_p\left(\frac{1+p^r}4\right)^2}{\G_p(\frac12)^2 \G_p\left(\frac{1+p^r}2\right)}\\
&=& 4^{\frac{p^r-p^{r-1}}4}\frac{\G_p\left(\frac{1+p^r}4\right)^2}{\G_p\left(\frac{1+p^r}2\right)}  \end{eqnarray*}
To continue, we apply Proposition \ref{prop:G14G12} to the last term to conclude {
\begin{eqnarray*}
\frac{\pFq{2}{1}{\frac{1-p^r}2\,\,\frac {1}2}{1}{2}}{\pFq{2}{1}{\frac{1-p^{r-1}}2\,\,\frac {1}2}{1}{2}}
&\equiv & (-1)^{\frac{p^2-1}8}
\left[1+\left(G_1\left(\frac12\right)-G_1\left(\frac14\right)\right)\frac{p^r}2
\right] \\
&&\cdot \frac{
\G_p\left(\frac14\right)^2}{\G_p\left(\frac12\right)} \frac{\left[1+G_1\left(\frac14\right)\frac{p^r}4\right]^2}{\left[1+G_1\left(\frac12\right)\frac{p^r}2\right]}\mod p^{2r}
\quad \text{}\\
&\equiv& (-1)^{\frac{p^2-1}8}\frac{ \G_p\left(\frac14\right)^2}{\G_p\left(\frac12\right)}
 \mod p^{2r}
\end{eqnarray*}}
the last congruence following as the coefficients of the $p^r$ terms cancel.
\end{proof}

\section{Proof of a strengthened conjecture of van Hamme and a Proposition}\label{ss:vanhamme}
We prove Theorem \ref{thm:2} in this section. We first prove a mod $p^6$ congruence between the $_7F_6$ in
question with a perturbed $_7F_6$. See \eqref{eq:7.2}. The perturbed version has a Galois symmetry with respect to
the $5$th roots of unity. This first congruence follows from somewhat involved  calculations and a use of Whipple's
well-posed $_7F_6$ evaluation formula. The perturbation we choose is suited
for using Dougall's formula to write our Galois symmetric perturbed $_7F_6$
expression as a $\G$-quotient and then a $\G_p$-quotient. The Galois symmetry in the $\G_p$-quotient
implies that in its $5$th degree Taylor expansion, the coefficients of $p^k$ for $1\leq k \leq 4$ vanish, so
we get a mod $p^5$ approximation. An extra factor of $p$ arises naturally
so in the end the congruence holds  mod $p^6$.

\subsection{A congurence with a perturbed $_7F_6$.   }
Let
$\zeta_5$ be any primitive $5$th root of unity. For $p \equiv 1 \mod 6$,
we have $\displaystyle \frac13 + \frac{p-1}3 = p\cdot \frac13$ so
 $\displaystyle\left(\frac 13\right)'=\frac 13$. For
$p\equiv 5\mod 6$ we have
$\displaystyle \frac13 + \frac{2p-1}3 = p\cdot \frac23$ so
 $\displaystyle\left(\frac 13\right)'=\frac 23$ in this case.
For convenience we write $\displaystyle
\left(\frac 13\right)'=
\frac {t}3$ where $t=1$ or $2$ as appropriate.
In either case
$\displaystyle \frac{1-tp}3 \in -\N$ so when this quantity appears in the top of an $_aF_b$ expression
the series terminates.

We are first going to show
\begin{multline}\label{eq:7.2}\pFq{7}{6}{\frac13\,\,\frac76\,\,\frac{1-\zeta_5 tp}3\,\, \frac{1-\zeta_5^2tp}3\,\, \frac{1-\zeta_5^3 tp}3\,\, \frac{1-\zeta_5^4tp}3\,\, \frac{1-tp}3}{\frac 16\,\,1+\frac{\zeta_5 tp}3\,\,
1+\frac{\zeta_5^2 tp}{3}\,\,1+\frac{\zeta_5^3 tp}{3}\,\, 1+\frac{\zeta_5^4 tp}{3}\,\, 1+\frac {tp}3}{1}\\\equiv \, \pFq{7}{6}{\frac13\,\,\frac76\,\,\frac13\,\,\frac13\,\,\frac13\,\,\frac13\,\,\frac13}{1\,\,1\,\,1\,\,1\,\,1\,\,1}{1}_{\frac{tp-1}3}  \equiv \, \pFq{7}{6}{\frac13\,\,\frac76\,\,\frac13\,\,\frac13\,\,\frac13\,\,\frac13\,\,\frac13}{1\,\,1\,\,1\,\,1\,\,1\,\,1}{1}_{{p-1}} \mod p^6.\end{multline}
The last congruence holds since
for $\frac {tp-1}3<k<p$ we have
$\left(\frac 13\right)_k \equiv 0 \mod p$
 and $(1)_k \not \equiv 0 \mod p$.

Consider a perturbed  terminating series  $$\,\pFq{7}{6}{\frac13,\frac76,\frac 13-\zeta_5 x,\frac 13-\zeta_5^2x,\frac 13-\zeta_5^3x,\frac 13-\zeta_5^4x,\frac{1-tp}3,}{\frac 16, 1+\zeta_5 x,1+\zeta_5 x^2, 1+\zeta_5^3 x, 1+\zeta_5^4 x,1+\frac {tp}3}{1}.$$ As a rational function of $x$, its coefficient ring is
$\Z[\zeta_5]$. However, it is Galois symmetric so   the coefficient ring is $\Z$.

\begin{lemma} \label{risingfraction}
 $\displaystyle
\frac{\left( \frac13-x\right)_k}{(1+x)_k}=\frac{\left(\frac13\right)_k}{(1)_k}[1+a_{k,1}x+a_{k,2}x^2+\cdots] \in \Z_p[[x]]$
for $\displaystyle k \le \frac{tp-1}3$.
\end{lemma}
\begin{proof}From
\begin{multline}\label{eq:oneonethird}
\left(\frac13-x\right)_k=\prod_{j=0}^{k-1}\left(\frac 13+j-x\right)=
\prod_{j=0}^{k-1}\left(\frac 13+j\right)\left(1-\frac{x}{\left(\frac 13+j\right)}\right)\\=
\left(\frac 13\right)_k \left [1-\sum_{j=0}^{k-1} \frac{3x}{1+3j}x+\sum_{0\le i< j\le k-1} \frac{9x^2}{(1+3i)(1+3j)}+\cdots \right ]\end{multline}
 we see that for $k$ in the specified range that none of the denominators
in~\eqref{eq:oneonethird} are multiples of $p$. It is trivial to see that for these $k$
 the constant term of the polynomial $(1+x)_k$ is {\bf not} divisible by $p$.
Thus its reciprocal, when viewed as a power series, has $p$-integral coefficients.

The last claim follows from Definition \ref{def:a'}.
\end{proof}
\begin{lemma}
With $a_{k,i}$ as in Lemma~\ref{risingfraction}
the congruence and equality below of  terminating power series in $x,y,z,w$ hold.
\begin{multline}\label{eq:morewhipple}
\,\pFq{7}{6}{\frac13\,\,\frac76\,\,\frac 13-x\,\,\frac 13-y\,\,\frac 13-z\,\,\frac 13-w\,\,\frac {1-tp}3}
{\frac 16\,\, 1+x\,\,1+y\,\, 1+z\,\, 1+w\,\,1+
\frac {tp}3}{1}\\ \equiv \pFq{7}{6}{\frac13\,\,\frac76\,\,\frac 13-x\,\,\frac 13-y\,\,\frac 13-z\,\,\frac 13-w\,\,\frac{1}3}
{\frac 16\,\, 1+x\,\,1+y\,\, 1+z\,\, 1+w\,\,1}{1}_{{\frac {tp-1}3}} \mod p  \\
 =\sum_{k=0}^{\frac{tp-1}3} (6k+1)\frac{(\frac 13)_k^6}{(1)_k^6} \times \left [1+\sum_{i\ge 1} a_{k,i}x^i \right ] \left [1+\sum_{i\ge 1} a_{k,i}y^i \right ]\left [1+\sum_{i\ge 1} a_{k,i}z^i\right ] \left [1+\sum_{i\ge 1} a_{k,i}w^i \right].
\end{multline}
Furthermore, the former series is in   $p\Z_p[[x,y,z,w]]$.
\end{lemma}
\begin{proof} The equality follows immediately from Lemma~\ref{risingfraction}  and the congruence is obvious.

To deal with the last claim,
we need  Whipple's well-posed $_7F_6$
evaluation formula (\cite[Theorem $3.4.5$]{AAR}) below
\begin{multline}\label{eq:whip}
\pFq{7}{6}{a\quad 1+\frac a2\quad b\quad c\quad d\quad e\quad f}{\frac12 a\,\, 1+a-b\,\,1+a-c\,\,1+a-d\,\,1+a-e\,\,1+a-f}{1}=\\
\frac{\G(1+a-d)\G(1+a-e)\G(1+a-f)\G(1+a-d-e-f)}{\G(1+a)\G(1+a-e-f)\G(1+a-d-e)\G(1+a-d-f)} \\ \times \,
\pFq{4}{3}{1+a-b-c\quad d\quad e\quad f}{d+e+f-a\quad 1+a-b\quad 1+a-c}{1}
\end{multline}
This holds when
the left side converges and the right side terminates, both of which happen. The $\G$-quotients are all rising factorials
of the form $(*)_{-f}$.

We apply \eqref{eq:whip} to
$\pFq{7}{6}{\frac13,\frac76,\frac 13-x,\frac 13-y,\frac 13-z,\frac 13-w,\frac {1-tp}3}{\frac 16, 1+x,1+y, 1+z, 1+w,1+\frac{tp}3}{1}$ to get
\begin{equation}\label{eq:anotherwhip}
\frac{ \left(\frac43\right)_{  {\frac{tp-1}3}} \left(\frac23+z+w\right)_{\frac{tp-1}3}}
{ (1+z)_{  {\frac{tp-1}3}  }  (1+w)_{  {\frac{tp-1}3}} }
\pFq{4}{3} {\frac23+x+y, \frac13-z,\frac13-w,\frac {1-tp}3,}
{\frac{2-tp}3-z-w,1+x, 1+y}{1}.
\end{equation}
Note $\displaystyle
\left(\frac43\right)_{ {\frac{tp-1}3} } =   \frac{tp}3
\left(\frac43\right)_{  {\frac{tp-4}3}} $.
The  rising factorials in the denominator of~\eqref{eq:anotherwhip} have constant terms that are units in $\Z_p$ as do all
terms in the denominators of the  terminating $_4F_3$ series.
\end{proof}

The coefficient of $x^5$ in \eqref{eq:morewhipple},
$\displaystyle
\sum_{k=0}^{  {\frac{tp-1}3}} (6k+1)\frac{\left(\frac 13\right)_k^6}{(1)_k^6} a_{k,5}$,
is in
$p\Z_p$.  Similarly, the coefficients of $x^4y$ etc. are in $p\Z_p$. We conclude that
$$\sum_{k=0}^{ {\frac{tp-1}3}}  (6k+1)\frac{\left(\frac 13\right)_k^6}{(1)_k^6} a_{k,4}a_{k,1},
\sum_{k=0}^{ {\frac{tp-1}3}} (6k+1)\frac{\left(\frac 13\right)_k^6}{(1)_k^6} a_{k,3}a_{k,2}$$ etc. are all in $p\Z_p$.

Letting $x=\zeta_5 u$, $y=\zeta_5^2u$, $z=\zeta_5^3u$ and $w=\zeta_5^4u$ above, set
\begin{multline}F(u):=
\,\pFq{7}{6}{\frac13,\frac76,\frac 13-\zeta_5 u,\frac 13-\zeta_5^2u,\frac 13-\zeta_5^3u,\frac 13-\zeta_5^4u,\frac{1}3-u,}{\frac 16, 1+\zeta_5 u,1+\zeta_5 u^2, 1+\zeta_5^3 u, 1+\zeta_5^4 u,1+u}
{1}_{\frac {tp-1}3}\\=C_0+\sum_{i > 0} C_{5i}u^{5i}\in \Z_p[[u^5]]
\end{multline}
with the equality following from
the symmetry with respect to all $5$th roots of unity.
If $C_5\in p\Z_p$, \eqref{eq:7.2} will follow
by setting $u=0$ and $ \displaystyle\frac{tp}3$ respectively.
The series expansion of $F(u)$ is
$$\sum_{k=0}^{{\frac{tp-1}3}} (6k+1)\frac{\left(\frac 13\right)_k^6}{(1)_k^6}
\times \prod_{j=0}^4 [1+\sum_{i\ge 1} a_{k,i}(\zeta_5^j u)^i].$$
Thus
$\displaystyle C_5=\sum_{k=0}^{{\frac{tp-1}3}} (6k+1)\frac{\left(\frac 13\right)_k^6}{(1)_k^6} [H(a_{k,i})]$
where $H(a_{k,i})$ in an integral polynomial in the $a_{k,i}$ where the second subscripts
in each monomial add to $5$. By the argument of the previous paragraph
$C_5 \in p\Z_p$ so~\eqref{eq:7.2} holds.

\subsection{ Dougall's formula, a $\G_p$-quotient, Galois symmetry, the proof of Theorem \ref{thm:2} and
a proposition. }
To prove Theorem~\ref{thm:2}
it remains to show  the left side of ~\eqref{eq:7.2}  is congruent to the right side in Theorem \ref{thm:2} mod
$ p^6$. We use
Dougall's formula (c.f. Theorem 3.5.1 of \cite{AAR} or  (8.2) of \cite{Whipple}) which
asserts that for  $f$  a negative integer and $1+2a=b+c+d+e+f$ that
\begin{multline}\label{eq:7.1}
\pFq{7}{6}{a\quad  1+\frac a2\quad b\quad c\quad d\quad e\quad f}{\frac a2\quad 1+a-b\quad 1+a-c\quad 1+a-d\quad 1+a-e\quad 1+a-f}{1}
\\
=\frac{(a+1)_{-f}(a-b-c+1)_{-f}(a-b-d+1)_{-f}(a-c-d+1)_{-f}}{(a-b+1)_{-f}(a-c+1)_{-f}(a-d+1)_{-f}(a-b-c-d+1)_{-f}}
\end{multline}
Set $\displaystyle a=\frac 13$, $\displaystyle b=\frac{1-\zeta_5 tp}3 $,
$\displaystyle c=
\frac{1-\zeta_5^2 tp}3 $, $\displaystyle d=\frac{1-\zeta_5^3 tp}3$,
$\displaystyle e=\frac{1-\zeta_5^4t p}3 $ and
$\displaystyle f=\frac{1-tp }3$. As $\zeta_5$ is a primitive $5$th root of unity, it
satisfies $1+\zeta_5+\cdots+\zeta_5^4=0$ and hence $1+2a=b+c+d+e+f$ so Dougall's formula  applies.
We get
\begin{multline}\label{eq:tp}
\pFq{7}{6}{\frac13\quad \frac76\quad \frac{1- \zeta_5 t p}3\quad \frac{1-\zeta_5^2t p}3\quad \frac{1-\zeta_5^3 tp}3\quad
\frac{1-\zeta_5^4tp}3\quad \frac{1-tp}3\quad }{\frac 16\quad  1+\frac{\zeta_5 tp}3\quad
1+\frac{\zeta_5^2 tp}3\quad  1+\frac{\zeta_5^3 tp}3\quad  1+\frac{\zeta_5^4 tp}3\quad 1+\frac{tp}3}{1}\\
=\frac{
\left(\frac43\right)_{ \frac{tp-1}3} \left( \frac{2+\zeta_5 tp +\zeta_5^2 tp}3 \right)_{ \frac{tp-1}3}
\left( \frac{2+\zeta_5 tp +\zeta_5^3 tp}3 \right)_{ \frac{tp-1}3}
\left( \frac{2+\zeta_5^2 tp +\zeta_5^3 tp}3 \right)_{ \frac{tp-1}3}
}
{ \left( 1+\frac{\zeta_5 tp}3 \right)_{\frac{tp-1}3}
\left( 1+\frac{\zeta_5^2 tp}3 \right)_{\frac{tp-1}3}
\left( 1+\frac{\zeta_5^3 tp}3 \right)_{\frac{tp-1}3}
\left( \frac{1+\zeta_5 tp+\zeta_5^2 tp+\zeta_5^3 tp}3 \right)_{\frac{tp-1}3}
 }\\
\end{multline}
Regardless of whether $t=1$ or $2$, none of the rising factorials in the denominator
contain a multiple of $p$. Also, when we switch to $\G_p$-quotients the power of $-1$ introduced
is
$(-1)^{ 8\frac{tp-1}3}=1$ and
observe
$$\left(\frac43\right)_{\frac{tp-1}3} =\frac{\G\left(\frac{tp+3}3\right)}{\G\left(\frac43\right)}
=\frac{\frac{tp}3}{\frac13} \frac{\G\left(\frac{tp}3\right)}{\G\left(\frac13\right)}=
 (-1)^{\frac{tp-1}3} tp\frac{\G_p\left(\frac{tp}3\right)}{\G_p\left(\frac13\right)}.$$

When $t=1$, that is $p \equiv 1 \mod 6$, the remaining rising factorials in the numerator contain no multiples
of $p$ and, when viewed as
$\Gamma$-quotients can be replaced directly by $\G_p$-quotients.

When $t=2$, that is $p \equiv 5 \mod 6$, the remaining rising factorials
in the numerator each contain exactly one multiple of $p$,
namely $\displaystyle \frac{2+*p}3+ \frac{p-2}3$. In this case, when replacing
the rising factorials by $\G_p$-quotients we have to include the factor
$$
\frac{p}3 (1+2\zeta_5 +2\zeta_5^2) \frac{p}3 (1+2\zeta_5 +2\zeta_5^3)  \frac{p}3 (1+2\zeta_5^2 +2\zeta_5^3)
= \frac{5p^3}{27}.
$$

Set $A_1=1$ and $A_2 = \displaystyle \frac{5p^3}{27}$.
Then~\eqref{eq:tp} becomes
\begin{multline}t
pA_t\frac{ \G_p\left(\frac{tp}3 \right)}{\G_p(\frac13)}
\frac{ \G_p\left( \frac{1+\zeta_5 tp+\zeta_5^2tp+tp}3\right)}{\G_p\left( \frac{2+\zeta_5 t p+\zeta_5^2tp}3\right)}
\frac{ \G_p\left( \frac{1+\zeta_5 tp+\zeta_5^3tp+tp}3\right)}{\G_p\left( \frac{2+\zeta_5 tp+\zeta_5^3tp}3\right)}
\frac{ \G_p\left( \frac{1+\zeta_5^2t p+\zeta_5^3tp+tp}3\right)}{\G_p\left( \frac{2+\zeta_5^2t p+\zeta_5^3tp}3\right)}\\
\times
\frac{ \G_p\left(1+\frac{\zeta_5 tp}3\right)}{ \G_p\left(\frac{2+\zeta_5 tp+tp}3\right)}
\frac{ \G_p\left(1+\frac{\zeta_5^2 tp}3\right)}{ \G_p\left(\frac{2+\zeta_5^2 tp+tp}3\right)}
\frac{ \G_p\left(1+\frac{\zeta_5^3 tp}3\right)}{ \G_p\left(\frac{2+\zeta_5^3 tp+tp}3\right)}
\frac{ \G_p\left(\frac{1+\zeta_5 tp+\zeta_5^2 tp +\zeta_5^3 tp }3\right)}
{ \G_p\left(\frac{\zeta_5 tp+\zeta_5^2 tp+\zeta_5^3 tp+tp}3\right)}\\
\end{multline}
As $1+\zeta_5 +\zeta_5^2 + \zeta_5^3  +\zeta_5^4 =0$, we have
$$\displaystyle\G_p\left(\frac{\zeta_5 tp+\zeta_5^2 tp+\zeta_5^3 tp+tp}3\right) =\G_p\left(-\frac{\zeta_5^4 tp}3\right)=(-1)^{a_0(0)}
\G_p\left(1+\frac{\zeta_5^4 tp}3\right)^{-1}$$
by the functional equation (see Proposition \ref{prop:valLB}). We apply the functional equation to the rest
of the denominator terms to get (grouping the $\G_p(1+*)$ terms at the end)
\begin{multline}(-1)^{p+6a_0(\frac23)}
tpA_t\frac{ \G_p\left(\frac{tp}3 \right)}{\G_p(\frac13)}
 \G_p\left( \frac{1+\zeta_5 t p+\zeta_5^2tp+tp}3\right)\G_p\left( \frac{1-\zeta_5 tp-\zeta_5^2tp}3\right)
 \G_p\left( \frac{1+\zeta_5 tp+\zeta_5^3tp+tp}3\right) \\
\times \G_p\left( \frac{1-\zeta_5 tp-\zeta_5^3tp}3\right) \G_p\left( \frac{1+\zeta_5^2 t p+\zeta_5^3tp+pt}3\right)\G_p\left( \frac{1-\zeta_5^2 tp-\zeta_5^3tp}3\right)
\G_p\left(\frac{1-\zeta_5 tp-tp}3\right)
{ \G_p\left(\frac{1-\zeta_5^2 tp-tp}3\right)}\\
\times
{ \G_p\left(\frac{1-\zeta_5^3 tp-tp}3\right)}
{ \G_p\left(\frac{1+\zeta_5 tp+\zeta_5^2 tp +\zeta_5^3 tp }3\right)}
\G_p\left(1+\frac{\zeta_5 tp}3\right){ \G_p\left(1+\frac{\zeta_5^2 tp}3\right)}
{ \G_p\left(1+\frac{\zeta_5^3 tp}3\right)}{ \G_p\left(1+\frac{\zeta_5^4tp}3\right)}\\
=(-1)^{p+6a_0(\frac23)}tpA_t\frac{ \G_p\left(\frac{tp}3 \right)}{\G_p(\frac13)}
 \G_p\left( \frac{1-\zeta_5^3 t p-\zeta_5^4tp}3\right)\G_p\left( \frac{1-\zeta_5 tp-\zeta_5^2tp}3\right)
 \G_p\left( \frac{1-\zeta_5^2 tp-\zeta_5^4tp}3\right) \\
\times \G_p\left( \frac{1-\zeta_5 tp-\zeta_5^3tp}3\right) \G_p\left( \frac{1-\zeta_5 t p-\zeta_5^4tp}3\right)\G_p\left( \frac{1-\zeta_5^2 tp-\zeta_5^3tp}3\right)
\G_p\left(\frac{1-\zeta_5 tp-tp}3\right)
{ \G_p\left(\frac{1-\zeta_5^2 tp-tp}3\right)}\\
\times
{ \G_p\left(\frac{1-\zeta_5^3 tp-tp}3\right)}
{ \G_p\left(\frac{1- tp-\zeta_5^4 tp }3\right)}
\G_p\left(1+\frac{\zeta_5 tp}3\right){ \G_p\left(1+\frac{\zeta_5^2 tp}3\right)}
{ \G_p\left(1+\frac{\zeta_5^3 tp}3\right)}{ \G_p\left(1+\frac{\zeta_5^4tp}3\right)}
\end{multline}
Note that $p+6a_0\left(\frac23\right)$ is odd and
$\displaystyle\G_p\left(\frac{tp}3\right) =-\G_p\left(1+\frac{tp}3\right)$.
Place this with the four terms at the end. We have symmetry with respect to the $5$th roots
of unity so
$$\G_p\left(1+\frac{\zeta_5 tp}3\right)
\G_p\left(1+\frac{\zeta_5^2 tp}3\right)
\G_p\left(1+\frac{\zeta_5^3 tp}3\right)
\G_p\left(1+\frac{\zeta_5^4tp}3\right)
\G_p\left(1+\frac{tp}3\right)$$
has,
by the $t=4, \,\, p\geq 11$ case of Theorem \ref{truncate},
 Taylor series expansion $\G_p(1)^5[1+O(p^5)]=-1+O(p^5)$.
The overall expression \eqref{eq:tp}  has the symmetry with respect to $5$th roots of unity as
does the remaining part,
\begin{multline} (-1)^{1+1}\frac{ tpA_t}{\G_p(\frac13)}
 \G_p\left( \frac{1-\zeta_5^3 t p-\zeta_5^4tp}3\right)\G_p\left( \frac{1-\zeta_5 tp-\zeta_5^2tp}3\right)
 \G_p\left( \frac{1-\zeta_5^2 tp-\zeta_5^4tp}3\right) \\
\times \G_p\left( \frac{1-\zeta_5 tp-\zeta_5^3tp}3\right) \G_p\left( \frac{1-\zeta_5 t p-\zeta_5^4tp}3\right)\G_p\left( \frac{1-\zeta_5^2 tp-\zeta_5^3tp}3\right)
\G_p\left(\frac{1-\zeta_5 tp-tp}3\right)
{ \G_p\left(\frac{1-\zeta_5^2 tp-tp}3\right)}\\
\times
{ \G_p\left(\frac{1-\zeta_5^3 tp-tp}3\right)}
{ \G_p\left(\frac{1-\zeta_5^4 tp- tp }3\right)}
\end{multline}
Thus above product has Taylor
 series expansion
$\displaystyle tpA_t\G_p\left(\frac13\right)^9[1+O(p^5)]$ as well. Multiplying by $-1+O(p^5)$  we get
\begin{eqnarray*}
\sum^{\frac{p-1}3}_{k=0} (6k+1) \frac{ \left(\frac13\right)^6_k}{(k!)^6}&=&\pFq{7}{6}{\frac13,\frac76,\frac{1}3,\frac{1}3,\frac{1}3,
\frac{1}3,\frac{1}3,}{\frac 16, 1,
1, 1, 1,1}{1}_{\frac{p-1}3}\\
&\equiv &\pFq{7}{6}{\frac13,\frac76,\frac{1- \zeta_5 p}3,\frac{1-\zeta_5^2 p}3,\frac{1-\zeta_5^3 p}3,
\frac{1-\zeta_5^4p}3,\frac{1-p}3,}{\frac 16, 1+\frac{\zeta_5 p}3,
1+\frac{\zeta_5^2 p}3, 1+\frac{\zeta_5^3 p}3, 1+\frac{\zeta_5^4 p}3,1+\frac{p}3}{1}\mod p^6\\
&\equiv &-tpA_t\G_p\left(\frac13\right)^9 \mod p^6.
\end{eqnarray*}
We have proved Theorem \ref{thm:2} for $p \geq 11$. We have  verified the cases $p=5,7$ by hand.
\hfill $\square$
\vskip1em
The next result is obtained by the same method as the previous proof. It uses a formula due to Pfaff-Saalsch\"utz which plays a role in proving
  \eqref{eq:7.1}  from \eqref{eq:whip}.
\begin{prop}\label{morethirds} For any prime $p>3$,
$$\pFq{3}{2}{\frac13,\frac13,\frac13}{1,1}{1}_{p-1}\equiv
\left\{\begin{array}{cll}\G_p\left(\frac13\right)^6& \mod p^3 &(\text{if } p\equiv 1\mod 6)\\
-\frac{p^2}{3}\G_p\left(\frac13\right)^6 & \mod p^3&(\text{if } p\equiv 5\mod 6) \end{array}\right. .$$
\end{prop}

\begin{proof} The Pfaff-Saalsch\"utz Theorem (Theorem 2.2.6 of \cite{AAR}) says for $n\in \mathbb N$
\begin{equation}\label{eq:4.12}
\pFq{3}{2}{-n,a,b}{c,1+a+b-c-n}{1}=\frac{(c-a)_n(c-b)_n}{(c)_n(c-a-b)_n}\,.
\end{equation}
As in the proof of Theorem~\ref{thm:2}, set $t=1$ or $2$ as $p \equiv 1$ or $5 \mod 6$.
Letting $\displaystyle n=\frac{tp-1}3,a=\frac{1-\zeta_3 tp}3$ and
$\displaystyle b=\frac{1-\zeta_3^2 tp}3$, where $\zeta_3$ is a primitive cube root of unity, we get
\begin{equation}\label{eq:zetathreetp}
\pFq{3}{2}{\frac{1-tp}3,\frac{1-\zeta_3 tp}3,\frac{1-\zeta^2_3 tp}3}{1,1}{1}=\frac{\left(\frac{2+\zeta_3tp}3\right)_{\frac{tp-1}3}\left(\frac{2+\zeta_3^2tp}3\right)_{\frac{tp-1}3 }}{(1)_{\frac{tp-1}3}\left(\frac{1- tp}3\right)_{\frac{tp-1}3}}.
\end{equation}
By the symmetry with respect to cube roots of unity
the left side, a finite sum,  agrees with
$\pFq{3}{2}{\frac{1}3,\frac{1}3,\frac{1}3}{1,1}{1}_{\frac{tp-1}3}$ mod $p^3$ and hence with
$\pFq{3}{2}{\frac{1}3,\frac{1}3,\frac{1}3}{1,1}{1}_{p-1}$
as the terms past $k=\displaystyle
\frac{tp-1}3$ in this series are all divisible by $p^3$.

When $p\equiv 1\mod 6$ ($t=1$) the rising factorials on the right side of \eqref{eq:zetathreetp} are all units in $\Z_p$. When $p\equiv 5\mod p$ ($t=2$), $\displaystyle
\left(\frac{2+\zeta_3tp}3\right)_{\frac{tp-1}3}$ and $\displaystyle
\left(\frac{2+\zeta_3^2tp}3\right)_{\frac{tp-1}3}$ contain, respectively, the multiples of $p$
$\displaystyle \frac{2+\zeta_3tp}3+\frac{p-2}3=\frac{2\zeta_3 +1}3p$ and $\displaystyle \frac{2+\zeta_3^2tp}3+\frac{p-2}3=\frac{2\zeta_3^2 +1}3p$.
Set $B_1=1$ and $B_2=\displaystyle
\frac{2\zeta_3+1}3p \cdot \frac{2\zeta_3^2+1}3p=\frac {p^2}3$.

We rewrite the right side of \eqref{eq:zetathreetp} using $\G_p$-quotients to  get
(there are four factors of $(-1)^{\frac{tp-1}3}$ which cancel)
\begin{multline}B_t \cdot \frac{\G_p\left(\frac{1+tp+\zeta_3tp}3\right)
 \G_p\left(\frac{1+tp+\zeta_3^2tp}3\right)\G_p(1)\G_p\left(\frac{1-tp}3\right) }
{\G_p\left(\frac{2+\zeta_3 tp}3\right) \G_p\left(\frac{2+\zeta_3^2 tp}3\right)
\G_p\left(\frac{2+ tp}3\right)  \G_p(0) }
=-B_t \cdot \frac{\G_p\left(\frac{1-\zeta_3^2tp}3\right) \G_p\left(\frac{1-\zeta_3tp}3\right)
\G_p\left(\frac{1-tp}3\right) }
{\G_p\left(\frac{2+\zeta_3 tp}3\right) \G_p\left(\frac{2+\zeta_3^2 tp}3\right)
\G_p\left(\frac{2+ tp}3\right)   }\\
=B_t (-1)^{1+3a_0(2/3) }\G_p\left(\frac{1-\zeta_3tp}3\right)^2 \G_p\left(\frac{1-\zeta_3^2tp}3\right)^2  \G_p\left(\frac{1-tp}3\right) ^2=
(-1)^{t+1}B_t\left(\G_p\left(\frac 13\right)^6 +O(p^3)\right)
\end{multline}
where the first equality uses that
$1+\zeta_3+\zeta_3^2=0$,  the second that
$\G_p\left(\frac 13+pb\right)\G_p\left(\frac 23-pb\right)
=(-1)^{a_0(2/3)}$ and the third  uses the symmetry with respect to the
cube roots of unity and that
$\displaystyle \frac23 \equiv \frac{p+2}3 \equiv 2\cdot\frac{p+1}3\mod p$
where $\displaystyle \frac{p+2}3 \in \N$ when $ p \equiv 1 \mod 6$ and
$\displaystyle 2\cdot\frac{p+1}3 \in \N$ when $ p \equiv 5 \mod 6$. Thus
$(-1)^{a_0(2/3)} = (-1)^{t}$.
\end{proof}

\section{The proof of Theorem \ref{thm:3} and Other applications}\label{ss:other}
We prove Theorem \ref{thm:3}.   We separate the proof  into the cases  where $p \equiv 3$ mod $4$
and $p \equiv 1$ mod $4$. For $p \equiv 3$ mod $4$ we write our $_3F_2$ as the constant term of a Taylor expansion of a perturbed
$_3F_2$ - actually we write the constant term in terms of the rest of the series and the perturbed
$_3F_2$.  We truncate the series to get a congruence
and using Propositions \ref{3f2series} and \ref{3f2gamma}
convert this congruence to one with $\G_p$-values, proving the result. The strategy for $p \equiv 1$ mod $4$ is
similar, but the computations are somewhat different.

In this section, we will consider supercongruences occurring for $\pFq{3}{2}{\frac12,\frac12,\frac12}{1,1}{\l}_{p-1}$
at values of $\l$ related to CM elliptic curves over $\Q$.
A brief description of the geometric background was given in the introduction.

\begin{lemma} \label{expansion}Let  $k \in {\mathbb N}$ with $k\le \displaystyle \frac{p-1}2$ and $M \in \Z$.
Set \begin{equation}\label{eq:AkBk}
     \displaystyle A_k =\sum^{k-1}_{j=0} \frac{1}{2j+1},\quad B_k=
\displaystyle \sum_{0\le i<j\le k-1}\frac1{(2i+1)(2j+1)}.
    \end{equation}
Note $A_k$ is defined for $k \geq 1$ and
$B_k$ is defined
for $k \geq 2$.
Then
$$\left(\frac{1-Mp}2\right)_k \equiv \left(\frac12\right)_k (1 - MpA_k+(Mp)^2B_k) \mod p^3$$
and
$$\left(1-Mp\right)_k \equiv (1)_k
(1 - MpE_k+(Mp)^2F_k) \mod p^3.$$
 There are similar
expressions for $E_k$ and $F_k$ that we will not need so we do not write them down.
\end{lemma}
\begin{proof}
Observe
$$
\left (\frac{1-Mp}2 \right )_k
=\left(\frac12-\frac{Mp}2\right)\left(\frac32-\frac{Mp}2\right)
\dotsb\left(\frac{2k-1}2-\frac{Mp}2\right).$$
Multiplying this out gives an expansion of the form
$$\left(\frac12\right)_k\left( 1-Mp\sum_{j=0}^{k-1}\frac1{2j+1}+(Mp)^2\sum_{0\le i<j\le
k-1}^{k-1}\frac1{(2i+1)(2j+1)}\right) +O(p^3).$$
The big `$O$' term is justified as no multiples of $p$ occur in any of the denominators
of any of the sums as all indices are at most
$\displaystyle \frac{p-3}2$.
The proof of the  $(1-Mp)_k$ result is identical.
\end{proof}

\begin{prop} \label{3f2series}
Set  $A_k$ and $B_k$ as in \eqref{eq:AkBk}.
 Set $_3F_2(C) =
\,\pFq{3}{2}{\frac{1-Cp}2,\frac{1+(C-2)p}2,\frac{1-p}2}{1-p,1-\frac{p}2}{1}.$
Then $$_3F_2(C)- \, _3F_2(D) \equiv \\ p^2
\sum^{\frac{p-1}2}_{k=0} \frac{ \left(\frac12\right)^3_k}{(k!)^3}
\left[ (2C^2-2D^2-4C+4D)B_k +(-C^2+D^2+2C-2D)A^2_k\right] \mod p^3.$$
\end{prop}
\begin{proof}
We apply Lemma~\ref{expansion} for $\displaystyle \frac12$ (resp. $1$)
with $M=C, 2-C$ and $1$ (resp. $M=1$ and $\displaystyle \frac12$).
So
\begin{eqnarray*}\label{eq:5.2}
_3F_2(C) &=&\,\pFq{3}{2}{\frac{1-Cp}2,\frac{1+(C-2)p}2,\frac{1-p}2}{1-p,1-\frac{p}2}{1}
=\sum^{\frac{p-1}2}_{k=0} \frac{ \left(\frac{1-Cp}2\right)_k  \left(\frac{1+(C-2)p}2\right)_k\left(\frac{1-p}2\right)_k}
{(1)_k\left(1-p\right)_k\left(1-\frac{p}2\right)_k}\\
&\equiv& \sum^{\frac{p-1}2}_{k=0} \frac{ \left(\frac12\right)^3_k
[1-CpA_k+C^2p^2B_k][1+(C-2)pA_k+(C-2)^2p^2B_k][1-pA_k+p^2B_k]}{(k!)^3
[1-pE_k+p^2F_k][1-\frac{p}2E_k+\frac{p^2}{4}F_k]}\\
&\equiv&
\sum^{\frac{p-1}2}_{k=0} \frac{ \left(\frac12\right)^3_k}{(k!)^3}
\left[1-3pA_k+p^2\left[(2C^2-4C+5)B_k+(-C^2+2C+2)A^2_k\right]\right]\times \\
&&
\left[1-\frac32 pE_k+\frac54p^2F_k
+\frac{p^2}2E^2_k\right]^{-1} \\
&\equiv&
\sum^{\frac{p-1}2}_{k=0} \frac{ \left(\frac12\right)^3_k}{(k!)^3}
\left[1+p\left[-3A_k+\frac32 E_k\right] + p^2 X\right] \mod p^3.
\end{eqnarray*} where $X=\left[(2C^2-4C+5)B_k +(-C^2+2C+2)A^2_k
 -\frac54 F_k +\frac74 E^2_k
-\frac92 A_kE_k\right]$.
The result follows by simply performing the subtraction $_3F_2(C) - \, _3F_2(D)$
as most terms are independent of $C$ and $D$ and therefore cancel.
\end{proof}

\begin{prop} \label{3f2gamma}
 Let $C,D>0$.

For $C\equiv D \equiv p \equiv 1 \mod 4$,
\begin{multline}_3F_2 ( C ) - \, _3F_2(D) \equiv  \left[-\frac{p^2}{16}\G_p\left(\frac14\right)^4\right] \\
{\bf \cdot}
\left[  (2C^2-2D^2-4C+4D)  G_2\left( \frac14\right)  +    (-2C^2+2D^2+4C-4D) G_1\left(\frac14\right)^2\right]
\mod p^3.
\end{multline}
Similarly,

when $C\equiv D \equiv p\equiv 3\mod 4$,
$$_3F_2 ( C ) - \, _3F_2(D) \equiv \frac{(C^2-D^2-2C+2D)}{16}p^2\G_p\left(\frac 14\right)^4 \mod p^4.$$

\end{prop}

\begin{proof} By Clausen's formula \eqref{eq:Clausen1}, which holds as long as both sides converge and the Gauss summation  formula \eqref{eq:Gauss} which applies  as
$\displaystyle\Re\left(1-\frac{p}2-\frac{1-Cp}4-\frac{1+(C-2)p}4)\right)=\Re\left(\frac 12
\right)=\frac 12>0$,
\begin{equation}\label{eq:5.5}\displaystyle\pFq{3}{2}{ \frac{1-Cp}2, \frac{1+(C-2)p}2, \frac{1-p}2}{1-p,1-\frac{p}2}{1}=
\, \pFq{2}{1}{ \frac{1-Cp}4, \frac{1+(C-2)p}4}{1-\frac{p}2}{1}^2 =
\frac{   \G\left( 1-\frac{p}2\right)^2 \G\left(\frac12\right)^2}
{\G\left( \frac{ 3+ (C-2)p}4\right)^2 \G\left(\frac{3-Cp}4\right)^2}.\end{equation}
In both pairs $\displaystyle\left(\frac12,\frac{3+(C-2)p}4\right)$ and
$\displaystyle\left(\frac{3-Cp}4, 1-\frac{p}2\right)$
the elements differ by an integer.  When $p\equiv 1\mod 4$, write $C=4k+1$. The  numbers between both
elements of the first pair that differ from each by an integer {\em and} are multiples
of $p$ are $\displaystyle\frac{2p}4, \frac{6p}4, \dots, \frac{4k-2}4 p$. For the second
pair the corresponding list is simply the negatives of the numbers above.
As in the variant of 3) of Lemma~\ref{GammaP}, when
we want to replace a quotient of $\G$-functions by the same quotient
of $\G_p$-functions, care must taken with the signs and multiples of $p$.
Here everything will be squared so we need not worry about signs, and the
multiples of $p$ of concern end up canceling in the numerator and denominator so we
may simply replace all $\G$-functions by $\G_p$.
Thus
$$\pFq{3}{2}{ \frac{1-Cp}2, \frac{1+(C-2)p}2, \frac{1-p}2}{1-p,1-\frac{p}2}{1}=
\frac{   \G\left( 1-\frac{p}2\right)^2 \G\left(\frac12\right)^2}
{\G\left( \frac{ 3+ (C-2)p}4\right)^2 \G\left(\frac{3-Cp}4\right)^2}
=
\frac{   \G_p\left( 1-\frac{p}2\right)^2 \G_p\left(\frac12\right)^2}
{\G_p\left( \frac{ 3+ (C-2)p}4\right)^2 \G_p\left(\frac{3-Cp}4\right)^2}$$
which by arguments which are now standard in this paper becomes
\begin{multline}
 \G_p\left(1-\frac{p}2\right)^2 \G_p\left(\frac12\right)^2
\G_p\left(\frac{1+(2-C)p}4\right)^2 \G_p\left( \frac{1+Cp}4\right)^2
\equiv \left[-\G_p\left(1-\frac{p}2\right)^2\G_p\left(\frac14\right)^4 \right]\\
{\bf \cdot} \left[ 1+G_1\left(\frac14\right) \left(\frac{2-C}4p\right) +\frac{G_2\left(\frac14\right)}2
\left(\frac{2-C}4 p\right)^2\right]^2
\left[ 1+G_1\left(\frac14\right) \left(\frac{C}4 p \right)\frac{G_2\left(\frac14\right)}2
\left(\frac{C}4 p\right)^2\right]^2 \\  \mod p^3.
\end{multline}
This simplifies to
$$-\G_p\left(1-\frac{p}2\right)^2\G_p\left(\frac14\right)^4
\left[ 1+pG_1\left(\frac14\right) +\frac{p^2}{16}Y\right]$$where $Y=\left[ (2C^2-4C+4)G_2\left(\frac14\right)
+(-2C^2+4C+4)G_1\left(\frac14\right)^2\right]$.
So
\begin{equation} _3F_2 ( C) - \, _3F_2(D) \equiv
\left(-\G_p\left(1-\frac{p}2\right)^2\G_p\left(\frac14\right)^4
\frac{p^2}{16} \right)  {\bf \cdot}
Y \mod p^3.
\end{equation}
By a simple application of Theorem~\ref{truncate} we see
$\displaystyle\G_p\left(1-\frac{p}2\right) \equiv \G_p(1)  \equiv -1 \mod p$.
The result follows for $p \equiv 1\mod 4$.

When $p\equiv 3 \mod 4$,  write $C=4k+3$, then within $\displaystyle\left (\frac 12,\frac{3+(C-2)p}4\right )$ the multiples of $p$ are $\displaystyle\frac {2p}4,\cdots, \frac{4k-2}4p$. The multiples of $p$  within
$\displaystyle\left (\frac{3-Cp}4,1-\frac p2 \right )$ include (up to sign) those just listed and one more,  $\displaystyle-\frac{(4k+2)p}4=-\frac{(C-1)p}4$. Thus
$$_3F_2(C)=\frac{(C-1)^2p^2}{4^2}\frac{   \G_p\left( 1-\frac{p}2\right)^2 \G_p\left(\frac12\right)^2}
{\G_p\left( \frac{ 3+ (C-2)p}4\right)^2 \G_p\left(\frac{3-Cp}4\right)^2}\equiv
\frac{(C-1)^2p^2\G_p\left(\frac 14\right)^4}{4^2} \mod p^3.$$ It follows that
$\displaystyle
_3F_2(C)-\,_3F_2(D)= \frac{(C-D)(C+D-2)p^2}{16}\G_p\left(\frac 14\right)^4 \mod p^3.$

\end{proof}

We are now ready to prove Theorem \ref{thm:3}.
Proposition \ref{3f2series} gives a formula for $_3F_2(C)-\, _3F_2(D)$ in terms of sums involving $A_k$ and $B_k$.
Proposition \ref{3f2gamma} uses Clausen's formula to change this to a $\G_p$-quotient. The  expansion \eqref{eq:3f2taylor}
 comes from Lemma \ref{expansion} and expresses our desired quantity in terms of a perturbed $_3F_2$ and an expression involving $A_k$ and $B_k$ that we have computed
as a $\G_p$-quotient. When $p \equiv 3$ mod $4$, the perturbed term is $0$ and we are done. For $p \equiv 1$ mod $4$ our standard approach applies: the perturbed term is reduced to a $\G$-quotient by use of Clausen's formula and
Gauss' summation formula. It is then changed to  a $\G_p$-quotient by Lemma \ref{GammaP}.

\begin{proof}[Proof of Theorem \ref{thm:3}]
We first handle the $p\equiv 3\mod 4$ case.
We need a vanishing result in this case.
Using the first equation on \cite[pp 142]{AAR} which holds when $n$ is a positive integer so  both sides converge
\begin{equation}\label{eq:57}
\pFq{3}{2}{-n,a,b}{d,e}{1}=\frac{(e-a)_n}{(e)_n}\pFq{3}{2}{-n,a,d-b}{d,a+1-n-e}{1}
\end{equation} we know when  $p\equiv 3\mod 4$,
\begin{equation}
\pFq{3}{2}{\frac{1-p}2,\frac{1+p}2,\frac 12}{1,1}{1}=(-1)^{\frac{p-1}2}\pFq{3}{2}{\frac{1-p}2,\frac{1+p}2,\frac 12}{1,1}{1}.
\end{equation} Thus $\pFq{3}{2}{\frac{1-p}2,\frac{1+p}2,\frac 12}{1,1}{1}=0$ in this case.

By Lemma~\ref{expansion}
\begin{equation}\label{eq:3f2taylor}
\pFq{3}{2}{\frac12,\frac{1}2,\frac{1}2}{1,1}{1} _{\frac{p-1}2}=
\pFq{3}{2}{\frac12,\frac{1-p}2,\frac{1+p}2}{1,1}{1} + p^2\sum_{k=0}^{\frac{p-1}2} \frac{\left(\frac 12\right)^3}{k!^3}[A_k^2-2B_k]\mod p^3.
\end{equation}
By Propositions~\ref{3f2series} and~\ref{3f2gamma}
 $$_3F_2(3)-\, _3F_2(7)\equiv 32p^2
\sum^{\frac{p-1}2}_{k=0}
\frac{\left(\frac 12\right)^3}{k!^3}[A_k^2-2B_k]\mod p^3$$  and
$$_3F_2(3)-\, _3F_2(7)\equiv -2p^2\G_p\left(\frac 14\right)^4 \mod p^3.$$
Thus $$\pFq{3}{2}{\frac12,\frac{1}2,\frac{1}2}{1,1}{1} _{\frac{p-1}2}\equiv 0+
\sum^{\frac{p-1}2}_{k=0}\frac{\left(\frac 12\right)^3}{k!^3}[A_k^2-2B_k]\equiv
-\frac{p^2}{16} \G_p\left(\frac 14\right)^4 \mod p^3.$$
This settles Theorem \ref{thm:3} for $p \equiv 3 \mod 4$.

Now we assume $p\equiv 1\mod 4$. Observe that by definition
$\displaystyle \pFq{3}{2}{\frac12,\frac12,\frac12}{1,1}{1}_{\frac{p-1}2}=
\sum_{k=0}^{\frac{p-1}2} \left(\frac{(\frac12)_k}{k!} \right)^3.$

Let $C=5$ and $D=9$ in Propositions~\ref{3f2series} and~\ref{3f2gamma}.
Taking care with signs, the differences $_3F_2(C)-_3F_2(D)$ are
\begin{equation}\label{eq:48}
48p^2\sum^{\frac{p-1}2}_{k=0} \left(\frac{ \left(\frac12\right)_k}{k!}\right)^3(2B_k-A^2_k) \equiv
-p^2\G_p\left(\frac14\right)^4 \left[ 6G_1\left(\frac14\right)^2-6G_2\left(\frac14\right)\right] \mod p^3.
\end{equation}

We will compute
$\pFq{3}{2}{\frac{1-p}2,\frac{1+p}2,\frac{1}2}{1,1}{1}$ in two different ways.
Since  $$\displaystyle \left( \frac{1-p}2+r\right) \left(\frac{1+p}2 +r\right) \left(\frac12+r\right)
=\left(\frac12+r\right)^3 - \frac{p^2}4\left(\frac12+r\right)$$
we have, employing the simple identity
$\displaystyle \sum^{k-1}_{j=0} \frac1{(2j+1)^2} = A^2_k-2B_k$,
$$\displaystyle \left( \frac{1-p}2\right)_k \left( \frac{1+p}2\right)_k\left( \frac{1}2\right)_k
\equiv \left( \frac{1}2\right)^3_k\left(1  -p^2 %\left(\frac12\right)^3_k
\sum^{k-1}_{j=0} \frac1{(2j+1)^2} \right) \equiv
\left( \frac{1}2\right)^3_k - p^2\left( \frac{1}2\right)^3_k (A^2_k-2B_k)
\mod p^3.$$ The first congruence below is just \eqref{eq:3f2taylor} while the second follows from
\eqref{eq:48}.
\begin{multline}\label{eq:second3f2}
\pFq{3}{2}{\frac{1-p}2\quad \frac{1+p}2\quad \frac{1}2}{1\quad 1}{1}\equiv
\,\pFq{3}{2}{\frac12\quad \frac12\quad \frac12}{1\quad 1}{1}_{\frac{p-1}2}-
p^2\sum_{k=0}^{\frac{p-1}2} \left (\frac{(\frac12)_k}{k!} \right )^3(A^2_k-2B_k)\\
\equiv \,\pFq{3}{2}{\frac12\quad \frac12\quad \frac12}{1\quad 1}{1}_{\frac{p-1}2}
-p^2\G_p\left(\frac14\right)^4 \left[
\frac{G_1\left(\frac14\right)^2-G_2\left(\frac14\right)}{8}\right] \mod p^3.
\end{multline}

To estimate the left side \eqref{eq:second3f2} we need
$p$-adic Gamma functions.
We use  Clausen's formula, Gauss' summation formula, that $p \equiv 1 \mod 4$
and note the elements of the pair  $\displaystyle\left(\frac{3-p}4,\frac12\right)$ differ by an integer,
as do the elements of   $\displaystyle\left(1,\frac{3+p}4\right)$.
In neither pair are there any  multiples of $p$ between the numbers and an integral distance
from both endpoints, so we can use 3) of Lemma~\ref{GammaP} directly to see
\begin{eqnarray*}\pFq{3}{2}{\frac{1-p}2,\frac{1+p}2,\frac{1}2}{1,1}{1}&=&\,
\pFq{2}{1}{\frac{1-p}4,\frac{1+p}4}{1}{1}^2
=\left[\frac{\G(1)\G\left(\frac12\right)}{\G\left(\frac{3-p}4\right)\G\left(\frac{3+p}4\right)}\right]^2\\
&=&\left[\frac{\G_p(1)\G_p\left(\frac12\right)}{\G_p\left(\frac{3-p}4\right)\G_p\left(\frac{3+p}4\right)}\right]^2
=-\G_p\left(\frac{1+p}4\right)^2\G_p\left(\frac{1-p}4\right)^2\\
&\equiv & -\G_p\left(\frac14\right)^4\left[1+\frac{G_2(\frac14)-G_1(\frac14)^2}{8}p^2\right] \\
&\equiv & -\G_p\left(\frac14\right)^4-p^2\G_p\left(\frac14\right)^4\left[\frac{G_2(\frac14)-G_1(\frac14)^2}{8}\right] \mod p^3.
\end{eqnarray*}

The fourth equality and first congruence involve what are for us standard arguments
with $\G_p$-functions.
Comparing with ~\eqref{eq:second3f2}, the $p \equiv 1$ mod $4$ case follows and
Theorem~\ref{thm:3} is proved.
\end{proof}

We prove one more supercongruence. We will restrict ourselves to the ordinary primes.
\begin{theorem}When $p\equiv 1\mod 4$ is a prime, then
$$\pFq{3}{2}{\frac12,\frac12,\frac12}{1,1}{\frac{-1}8}_{\frac{p-1}2}\equiv\, -(-1)^{\frac{p^2-1}8}\G_p\left(\frac14\right)^4 \mod p^3.$$
\end{theorem}
\begin{proof}Here, we will outline the proof and will skip some details if the analysis is similar to what we have used above.

When $a$ is not a negative integer, \eqref{eq:Ka} says
\begin{equation}\label{eq:Ka1}
\pFq{3}{2}{3a-1,a,1-a}{,2a,a+\frac12}{-\frac18}=2^{3a-3}\left (\frac{\G(a+\frac12)\G(\frac a2)}{\G(\frac {3a}2)\G(\frac 12)} \right )^2\,.
\end{equation}

Setting $a=\frac{1+p}2$ in~\eqref{eq:Ka1}
\begin{equation}
 \pFq{3}{2}{\frac{1+3 p}2\quad \frac{1+p}2\quad \frac{1-p}2}{1+p\quad 1+ \frac p2}{\frac{-1}8}=
2^{\frac{3p-3}2}
\left ( \frac{\G(1+\frac p2)\G(\frac {1+p}4)}{\G(\frac {3+3p}4)\G(\frac 12)} \right )^2
\end{equation}
The number $\displaystyle \frac12$ and $\displaystyle 1+\frac p2$ differ by an integer and $\displaystyle \frac p2$
is the only multiple of $p$ an integral distance from both  of them. The same holds for $\displaystyle \frac{1+p}4$
and $\displaystyle \frac{3+3p}4$ so
$$ 2^{\frac{3p-3}2}
\left ( \frac{\G(1+\frac p2)\G(\frac {1+p}4)}{\G(\frac {3+3p}4)\G(\frac 12)} \right )^2
= 2^{\frac{3p-3}2}\left ( \frac{\G_p(1+\frac p2)\G_p(\frac {1+p}4)}{\G_p(\frac {3+3p}4)\G_p(\frac 12)} \right )^2.$$

Expanding the above using local analytic property of $\G_p(\cdot)$ and \eqref{eq:4power2}, we have
\begin{equation}\label{eq:5.12}
\pFq{3}{2}{\frac{1+3 p}2,\frac{1+p}2,\frac{1-p}2}{1+p, 1+ \frac p2}{\frac{-1}8}\equiv (-1)^{\frac{p^2-1}8}\frac{\G_p(\frac 14)^4}{\G_p(\frac 12)^2} \left [ 1+\frac 32Xp+\frac18 Y p^2\right ] \mod p^3.
\end{equation}
where
$$\displaystyle X=G_1(0)-G_1\left(\frac 14\right) \quad \text{and} \quad
\displaystyle Y=-18G_1(0)G_1\left(\frac14\right)+4G^2_1\left(\frac14\right)+5G_2\left(\frac14\right)+9G^2_1(0).
$$

Letting  $\displaystyle a=\frac{1- p/3}2$  in \eqref{eq:Ka1},

\begin{equation}\label{eq:4.6}\pFq{3}{2}{\frac{1- p}2,\frac{1+p/3}2,\frac{1-p/3}2}{1-\frac p3, 1- \frac p6}{\frac{-1}8}=2^{\frac{-3- p}2}\left ( \frac{\G(1- \frac p6)\G(\frac{1- p/3}4)}{\G(\frac{3- p}4)\G(\frac12)}\right )^2.\end{equation}
When $p\equiv 1\mod 4$, one of $1-\frac p6$ and $\frac{1-p/3}4$ is in $\frac 16+\Z$ and the other one is in $\frac 56+\Z$. Thus $\frac{\G(\frac16)\G(\frac56)}{\G(1- \frac p6)\G(\frac{1- p/3}4)}$ is a product of 2 rising factorials and $\frac{\G(\frac 12)}{\G(\frac{3- p}4)}$ is itself a rising factorial. None of them contains multiples of $p$. Consequently, $$ \left ( \frac{\G(1- \frac p6)\G(\frac{1- p/3}4)}{\G(\frac{3- p}4)\G(\frac12)}\right )^2=\left (
\frac{\G(1- \frac p6)\G(\frac{1- p/3}4)}{\G(\frac16)\G(\frac56)}\frac{\G(\frac 12)}{\G(\frac{3- p}4)}\frac{\G(\frac16)\G(\frac56)}{\G(\frac12)^2} \right )^2=4\left ( \frac{\G_p(1- \frac p6)\G_p(\frac{1- p/3}4)\G_p(\frac 12)}{\G_p(\frac16)\G_p(\frac56)\G_p(\frac{3- p}4)}\right )^2$$ as $\left (\frac{\G(\frac 16)\G(\frac 56)} {\G(\frac 12)^2}\right )^2=4$. So,
\begin{equation}\label{eq:for35}
\pFq{3}{2}{\frac{1- p}2,\frac{1+p/3}2,\frac{1-p/3}2}{1-\frac p3, 1- \frac p6}{\frac{-1}8}=2^{\frac{1- p}2}\left ( \frac{\G_p(1- \frac p6)\G_p(\frac{1- p/3}4)\G_p(\frac 12)}{\G_p(\frac16)\G_p(\frac56)\G_p(\frac{3- p}4)}\right )^2.\end{equation}
Together using \eqref{eq:4power2}, we have
\begin{equation}\label{eq:5.15}
\pFq{3}{2}{\frac{1- p}2,\frac{1+p/3}2,\frac{1-p/3}2}{1-\frac p3, 1- \frac p6}{\frac{-1}8}\equiv (-1)^{\frac{p^2-1}8}\frac{\G_p(\frac 14)^4}{\G_p(\frac 12)^2} \left [ 1-\frac12 Xp+\frac{1}{72}Yp^2 \right ] \mod p^3
\end{equation}where $X$ and $Y$ as before.

When $a=\displaystyle \frac{1-p}2$, by\eqref{eq:Ka2}

\begin{equation}\label{eq:Ka2}
\pFq{3}{2}{\frac{1-3p}{2}, \frac{1-p}2,\frac{1+p}2}{1-p,1-\frac p2}{-\frac 18}=2^{\frac{3(p-1)}2} \left( {\G_p\left(\frac 12\right)\G_p(1-p)}{\G_p\left(\frac{1+3p}3\right)\G_p\left(\frac{1+p}4\right)}\right)^2
\end{equation}
Applying \eqref{eq:4power2} to $2^{\frac{3(p-1)}2}$ and the analytic expansion to $\G_p(\cdot)$, we have

\begin{equation}\label{eq:5.17}
\pFq{3}{2}{\frac{1-3p}{2}, \frac{1-p}2,\frac{1+p}2}{1-p,1-\frac p2}{-\frac 18}\equiv (-1)^{\frac{p^2-1}8}\frac{\G_p(\frac 14)^4}{\G_p(\frac 12)^2} \left [ 1-\frac 32 Xp+\frac18 Yp^2 \right ] \mod p^3
\end{equation}
\vskip1em

Now we rewrite the right sides of \eqref{eq:5.12}, \eqref{eq:5.15}, and \eqref{eq:5.17} in terms harmonic sums.  So mod $p^3$ we get a linear system of 3 equations of the form
$$A_0+iA_1p+i^2A_2p^2\equiv (-1)^{\frac{p^2-1}8}\frac{\G_p(\frac 14)^4}{\G_p(\frac 12)^2}
\left [ 1+i Xp+\frac{i^2}{18} Yp^2 \right ] \mod p^3$$ with $i=3/2,-1/2,-3/2$ and
$\displaystyle A_0=\sum_{k=0}^{\frac{p-1}2} \left (\frac{(\frac 12)_k}{k!} \right ) \left (\frac{-1}8 \right )^k$ and $A_1,A_2$ involving harmonic sums like \eqref{eq:AkBk}.
When  $p\ge 5$  the matrix $M=\begin{pmatrix}1&3/2&1/8\\1&-1/2&1/72\\1&-3/2&1/8\end{pmatrix}$ is invertible in $\Z/p^3\Z$, thus
$$\displaystyle
A_0\equiv (-1)^{\frac{p^2-1}8}\frac{\G_p\left(\frac 14\right)^4}{\G_p\left(\frac12\right)^2} =
-(-1)^{\frac{p^2-1}8}\G_p\left(\frac 14\right)^4 \mod p^3.$$
\end{proof}

\section{Appendix: Hypergeometric formulae}\label{ss:Appendix}
\subsection{Pfaff transformation formula}  \cite[(2.3.14), pp. 79]{AAR}
\begin{equation}\label{eq:pfaff}
\pFq{2}{1}{-n,b}{c}{x}=\frac{(c-b)_n}{(c)_n} \, \pFq{2}{1}{-n,b}{,b+1-n-c}{1-x}.
\end{equation}This equation is derived by finding solutions to a differential
equation with three singularities. For general $n,b$ and $c$
~\eqref{eq:pfaff}
 holds for $x \in \C \backslash [1,\infty)$.
Here, as $n \in \N$, both sides are polynomials and equality on $\C \backslash [1,\infty)$ implies
equality on all of $\C$, and in particular at $x=2$, the value we will need. In \cite{AAR} the proof
of $(2.3.14)$ involves an argument where $c$ is assumed {\em not} in $\Z$ unless otherwise stated. This
pertains only to the number of independent {\em analytic} solutions to the differential
equation at hand and does not intervene here.
\subsection{Clausen formula}    The following formula holds as formal power series and hence holds when both hand sides converge. See \cite[\S2.5]{Slate} and Page 116 of \cite{AAR}.
\begin{equation}\label{eq:Clausen1}
\left( \pFq{2}{1}{a,b}{,a+b+\frac12}{x}\right)^2=\, \pFq{3}{2}{2a,2b,a+b}{,2a+2b,a+b+\frac12}{x}.
\end{equation}

\subsection{Some evaluation formulae with argument $z=1$}
\noindent\newline
\begin{itemize}
\item Gauss summation formula. Thereof 2.2.2 of \cite{AAR}
For $\Re(c-a-b)>0$,
\begin{equation}\label{eq:Gauss}
\pFq{2}{1}{a,b}{,c}{1}=\frac{\G(c)\G(c-a-b)}{\G(c-a)\G(c-b)}
\end{equation}

\item Chu-Vandermonde Theorem (Corollary 2.2.3 of \cite{AAR}). For $n\in \mathbb N$
\begin{equation}\label{eq:CV}
\pFq{2}{1}{-n,a}{,c}{1}=\frac{(c-a)_n}{(c)_n}\,.
\end{equation}\label{eq:PS}
\item Pfaff-Saalsch\"utz Theorem (Theorem 2.2.6 of \cite{AAR}). For $n\in \mathbb N$
\begin{equation}
\pFq{3}{2}{-n,a,b}{,c,1+a+b-c-n}{1}=\frac{(c-a)_n(c-b)_n}{(c)_n(c-a-b)_n}\,.
\end{equation}
\item  Dixon formula ((2.2.11) of \cite{AAR}). For $\Re\left({\frac{a}2+b+c+1}\right)>0$,
\begin{equation}\label{eq:Dixon}
\pFq{3}{2}{a,-b,-c}{,a+b+1,a+c+1}{1}=
\frac{\G(\frac a2+1)\G(a+b+1)\G(a+c+1)\G(\frac a2+b+c+1)}{\G(a+1)\G(\frac a2+b+1)\G(\frac a2+c+1)\G(a+b+c+1)}\,.
\end{equation}
\item The next formula   holds when $n$ is a positive integer and both sides converge
\begin{equation}\label{eq:57}
\pFq{3}{2}{-n,a,b}{d,e}{1}=\frac{(e-a)_n}{(e)_n}\pFq{3}{2}{-n,a,d-b}{d,a+1-n-e}{1}
\end{equation}  For a more detailed argument:
\eqref{eq:57}  is derived from Theorem $3.3.3$ of \cite{AAR} which in turn
is derived from Euler's transformation formula (Theorem $2.2.5$ of \cite{AAR})
which follows from Euler's Integral formula (Theorem $2.2.1$ of \cite{AAR}) which holds
on an open set. Both sides of~\eqref{eq:57} are rational functions of $a,b,d,e$ and as
they agree on
 some open subset of ${\mathbb R}^4$ they agree
everywhere.

\item Whipple's well-posed $_7F_6$
evaluation formula, Theorem $3.4.5$ of \cite{AAR}
\begin{multline}
\pFq{7}{6}{a\quad 1+\frac a2\quad b\quad c\quad d\quad e\quad f}{\frac12 a\quad 1+a-b\quad 1+a-c\quad 1+a-d\quad 1+a-e\quad 1+a-f}{1}=\\
\frac{\G(1+a-d)\G(1+a-e)\G(1+a-f)\G(1+a-d-e-f)}{\G(1+a)\G(1+a-e-f)\G(1+a-d-e)\G(1+a-d-f)} \\\times \,
\pFq{4}{3}{1+a-b-c,d,e,f}{d+e+f-a,1+a-b,1+a-c}{1}
\end{multline}
This holds when
the left side converges and the right side terminates.
\item Dougall's formula (c.f. Theorem 3.5.1 of \cite{AAR} or  (8.2) of \cite{Whipple}) which
asserts that for  $f$  a negative integer and $1+2a=b+c+d+e+f$ then
\begin{multline}
\pFq{7}{6}{a\quad 1+\frac a2\quad b\quad c\quad d\quad e\quad f}{\frac a2\quad 1+a-b\quad 1+a-c\quad 1+a-d\quad 1+a-e\quad 1+a-f}{1}=
\\\frac{(a+1)_{-f}(a-b-c+1)_{-f}(a-b-d+1)_{-f}(a-c-d+1)_{-f}}{(a-b+1)_{-f}(a-c+1)_{-f}(a-d+1)_{-f}(a-b-c-d+1)_{-f}}
\end{multline}
\end{itemize}
\subsection{Some evaluation formulae with argument $z\neq 1$}

\begin{itemize}
\item Kummer's formula. Corollary 3.1.2 of \cite{AAR} or (2.3.2.9) of \cite{Slate}

\begin{equation}\label{eq:Kummer}
\pFq{2}{1}{a,b}{,a-b+1}{-1}=\frac{\G(a-b+1)\G(\frac a2+1)}{\G(a+1)\G(\frac a2-b+1)}.
\end{equation} This holds when $\Re(b)<1$ and both  sides converge. In particular, it holds when $b$ is a negative integer and the left  side converges. In this case, the right  side can be written as $\frac{(a+1)_{-b}}{(1+\frac a2)_{-b}}$.
\item  If  $a\notin \Z_{<0}$, then
\begin{equation}\label{eq:Ka}
\pFq{3}{2}{3a-1,a,1-a}{,2a,a+\frac12}{-\frac18}=2^{3a}\left (\frac{\G\left(a+\frac12\right)\G\left(\frac a2\right)}{\G\left(\frac {3a}2\right)
\G\left(\frac 12\right)} \right )^2
\end{equation}This is entry 1 in Table 1 of  \cite{Ka2} by Karlsson. When $a\in \Z_{<0}$, the value is different. In particular, when $p\equiv 1\mod 4$ and  $a=\frac{1-p}2$,
\begin{equation}\label{eq:Ka2}
\pFq{3}{2}{3a-1,a,1-a}{,2a,a+\frac12}{-\frac18}=-2^{\frac{3(p-1)}2}
\left( {\G_p(1-p)}{\G_p\left(\frac{1+3p}4\right)\G_p\left(\frac{1+p}4\right)}\right)^2.
\end{equation}
 To see why there are two cases,
we first recall another formula of Pfaff
\begin{equation}\label{eq:pfaff2}
\pFq{2}{1}{a,b}{c}{z}= \left \{ \begin{array}{ll}
(1-z)^{-b}\, \pFq{2}{1}{a,c-b}{c}{\frac{z}{z-1}}        & \text{ if }  b\in \Z_{<0} \text{ and } b<\Re(a)\\(1-z)^{-a}\, \pFq{2}{1}{a,c-b}{c}{\frac{z}{z-1}} & \text{ otherwise.}                       \end{array} \right.
\end{equation} For the first case, see \cite[pp 79]{AAR}, which can be verified using our argument below.  The second case can be deduced from the argument in \cite[pp 76]{AAR}. In articular, we note that in this case, both hand sides are polynomials in $z$.
We  also need a quadratic formula (which can be derived from the quadratic formula on pp 130, right above (3.1.16) \cite{AAR} and the above Pfaff formula \eqref{eq:pfaff2})
\begin{equation}\label{eq:Kummer2}
\pFq{2}{1}{\frac a2,\b-\frac{a}2}{b+\frac 12}{\frac{z^2}{4(z-1)}}= \left \{ \begin{array}{ll}  (1-z)^{\frac{b}2}\pFq{2}{1}{a,b}{2\b}{z}& \text{ if } b\in \Z_{<0} \text{ and } b<\Re(a)\\(1-z)^{\frac{a}2}\pFq{2}{1}{a,b}{2\b}{z} &\text{ otherwise.}  \end{array}\right.
\end{equation} Another ingredient we need is how to deduce Kummer's formula \eqref{eq:Kummer} from Gauss summation \eqref{eq:Gauss}. This formula appears on page 126 of \cite{AAR} which says
\begin{equation}\label{eq:GtoK}
\pFq{2}{1}{a,b}{1+a-b}{-1}=2^{-a} \pFq{2}{1}{\frac a2, \frac{a+1}2-b}{1+a-b}{1}.
\end{equation}

Now we will show case 2 of \eqref{eq:Ka}.  When $p\equiv 1\mod 4$ and $a=\frac{1-p}2$,

\begin{eqnarray*}
\pFq{3}{2}{\frac{1-3p}2,\frac{1-p}2,\frac{1+p}2}{1-p,1-\frac p2}{-\frac18}&\overset{ \eqref{eq:Clausen1}}{=}&\, \pFq{2}{1}{\frac{1-3p}4,\frac{1+p}4}{1 -\frac p2}{-\frac18}^2\\
&\overset{\eqref{eq:Kummer2}}{=}&\left (\left(1-\frac 12\right)^{\frac{p-1}4}\pFq{2}{1}{\frac{1+p}2,\frac {1-p}2}{1-p}{\frac 12}\right )^2\\
&\overset{\eqref{eq:pfaff2}}{=}& \left (\left(1-\frac 12\right)^{\frac{p-1}4}(1+1)^{\frac{p-1}2}\pFq{2}{1}{\frac{1+p}2,\frac {1-p}2}{1-p}{-1}\right )^2\\
&=&\left (2^{\frac{(p-1)}4}\pFq{2}{1}{\frac{1+p}2,\frac {1-p}2}{1-p}{-1}\right )^2\\
&\overset{\eqref{eq:GtoK}}{=}&\left (2^{\frac{(p-1)}4}2^{\frac{(p-1)}2}\pFq{2}{1}{\frac{1-p}4,\frac{1-3p}4 }{1-p}{1} \right)^2\\
&\overset{\eqref{eq:CV}}{=}& 2^{\frac{3(p-1)}2} \left( \frac{\left(\frac{3-p}4\right)_{\frac{p-1}4}}{(1-p)_{\frac{p-1}4}}\right)^2\\
&\overset{\text{ 3) of Lemma }\, \ref{GammaP}}{=}& 2^{\frac{3(p-1)}2}
\left( \frac{\G_p\left(\frac 12\right)\G_p(1-p)}{\G_p\left(\frac{3-p}4\right)\G_p\left(\frac{3-3p}4\right)}\right)^2\\
&=& -2^{\frac{3(p-1)}2} \left( {\G_p(1-p)}
{\G_p\left(\frac{1+3p}4\right)\G_p\left(\frac{1+p}4\right)}\right)^2
\end{eqnarray*}

\end{itemize}


\begin{thebibliography}{B }
\bibitem{Ahlgren} S.~Ahlgren, \emph{Gaussian hypergeometric series and combinatorial
     congruences}, Symbolic computation, number theory, special functions, physics and combinatorics,
Dev. Math 4, Kluwer, Dodrecht 2001, 1--12.
\bibitem{AO}S.~Ahlgren, and K.~Ono,  \emph{A Gaussian hypergeometric series evaluation and Ap\'ery number congruences}. J. Reine Angew. Math. 518 (2000), 187--212.
\bibitem{AOP}S.~Ahlgren, K.~Ono, and D.~Penniston,  \emph{Zeta functions of an infinite family of K3 surfaces}. Amer. J. Math. 124 (2002), no. 2, 353--368.
\bibitem{AAR} G.~Andrews, R.~Askey, and R.~Roy,  \emph{Special Functions}. Cambridge University Press, (1999).
\bibitem{BvS}V.V~Batyrev and D.~van Straten, Duco \emph{Generalized hypergeometric functions and rational curves on Calabi-Yau complete intersections in toric varieties}. Comm. Math. Phys. 168 (1995), no. 3, 493--533
\bibitem{Beukers1} F. Beukers, \emph{ Some congruences for the Ap\'ery numbers.} J. Number Theory 21 (1985), no. 2, 141--155.
\bibitem{Beukers2} \bysame, \emph{Another congruence for the Ap\'ery numbers.} J. Number Theory 25 (1987), no. 2, 201--210.

\bibitem{Chan} H.-H.~Chan, S.~Cooper, Shaun, and F.~Sica, Francesco,
\emph{Congruences satisfied by Ap\'ery-like numbers}.
Int. J. Number Theory 6 (2010), no. 1, 89--97.
\bibitem{CKKO} H.-H.~Chan, A.~Kontogeorgis, C.~Krattenthaler, and R.~Osburn, \emph{Supercongruences satisfied by coefficients of 2F1 hypergeometric series.} Ann. Sci. Math. Québec 34 (2010), no. 1, 25--36.
\bibitem{Cohen} H.~Cohen, \emph{Number Theory}. Vol. II. Analytic and modern tools. Graduate Texts in Mathematics, 240. Springer, New York, 2007. xxiv+596 pp.
\bibitem{Coster} M.~Coster,
\emph{Generalisation of a congruence of Gauss.}
J. Number Theory 29 (1988), no. 3, 300--310.
\bibitem{CvH}M.~Coster and L.~ van Hamme,
\emph{Supercongruences of Atkin and Swinnerton-Dyer type for Legendre polynomials}. J. Number Theory 38 (1991), no. 3, 265--286.
\bibitem{Dwork-pcycle}B.M.~Dwork, \emph{ $p$-adic Cycles,} \emph{Publications math\'{e}matiques de l'I.H.\'{E}.S.},   37, (1969), 27--115.
\bibitem{Dwork} \bysame, \emph{Lectures on $p$-adic differential equations}. With an appendix by Alan Adolphson. Grundlehren der Mathematischen Wissenschaften [Fundamental Principles of Mathematical Science], 253. Springer-Verlag, New York-Berlin, 1982. viii+310 pp.
\bibitem{Ka2}P.W.~Karlsson,  \emph{ Clausen's hypergeometric function with variable -1/8 or -8}. Math. Sci. Res. Hot-Line 4 (2000), no. 7, 25--33.
\bibitem{KLMSY}
J.~Kibelbek, L.~Long, K.~Moss, B.~Sheller, and H.~Yuan.
\newblock \emph{Supercongruences and complex multiplication.} J. Number Theory, to appear,
\newblock arXiv 1210:4489.
\bibitem{Kil} T.~Kilbourn,  \emph{An extension of the Ap\'ery number supercongruence}. Acta Arith. 123 (2006), no. 4, 335--348.
\bibitem{Long}L.~Long, \emph{Hypergeometric evaluation identities and supercongruences}. Pacific J. Math. 249 (2011), no. 2, 405--418.
\bibitem{M}Y.~Morita,  \emph{A $p$-adic analogue of the $\G$-function}. J. Fac. Sci. Univ. Tokyo Sect. IA Math. 22 (1975), no. 2, 255--266.
\bibitem{McO}
D.~McCarthy,  and R.~Osburn,
\emph{A $p$-adic analogue of a formula of Ramanujan},  Arch. Math. (Basel) 91 (2008), no. 6, 492--504.
\bibitem{MO}E.~Mortenson, \emph{Supercongruences for truncated $_{n+1}F_n$ hypergeometric series with applications to certain weight three newforms}. Proc. Amer. Math. Soc. 133 (2005), no. 2, 321--330.
\bibitem{OSS}R.~Osburn, B.~Sahu, and  A.~Straub \emph{Supercongruences for sporadic sequences},  	 arXiv:1312.2195, (2013).

\bibitem{RV}F.~Rodriguez-Villegas, \emph{Hypergeometric families of Calabi-Yau manifolds. Calabi-Yau varieties and mirror symmetry} (Toronto, ON, 2001), 223--231, Fields Inst. Commun., 38, Amer. Math. Soc., Providence, RI, 2003.
\bibitem{RZ95} A.~Robert and  M.~Zuber,
\emph{The Kazandzidis supercongruences. A simple proof and an application.}
Rend. Sem. Mat. Univ. Padova 94 (1995), 235--243.
\bibitem{SC} A.~Selberg, and S.~Chowla,  \emph{On Epstein's zeta-function}. J. Reine Angew. Math. 227 1967 86--110.
\bibitem{SB} J.~Stienstra and F.~Beukers,  \emph{On the Picard-Fuchs Equation and the Formal Brauer Group of Certain Elliptic $K3$-Surfaces}. \emph{Mathematische Annalen}, 271, no.2, (1985), 269-304.
\bibitem{Slate}L.J.~ Slater,  \emph{Generalized hypergeometric functions}. Cambridge University Press, Cambridge (1966).% xiii+273 pp

\bibitem{Sun1} Z.-W. Sun, \emph{On congruences related to central binomial coefficients}, J. Number Theory 131 (2011) 2219--2238.
\bibitem{Sun2} \bysame, \emph{Super congruences and Euler numbers}, Sci. China Math. 54 (2011) 2509--2535.
\bibitem{vanHamme}L.~van Hamme,
\emph{Some conjectures concerning partial sums of generalized hypergeometric series}. $p$-adic functional analysis (Nijmegen, 1996), 223--236,
Lecture Notes in Pure and Appl. Math., 192, Dekker, New York, 1997.
\bibitem{Whipple}
F.J.W.~Whipple, \emph{On well-posed series, generalised hypergeometric
series
  having parameters in pairs, each pair with the same sum}, Proc. London Math.
  Soc. {24} (1926), no.~2, 247--263.
\bibitem{Zudilin} W.~Zudilin,  \emph{Ramanujan-type supercongruences.} J. Number Theory 129 (2009), no. 8, 1848--1857.
\end{thebibliography}
\end{document}